\PassOptionsToPackage{unicode}{hyperref}
\PassOptionsToPackage{hyphens}{url}
\documentclass[
  12pt,
]{article}
\usepackage{amsmath,amssymb}
\usepackage{lmodern}
\usepackage{ifxetex,ifluatex}
\ifnum 0\ifxetex 1\fi\ifluatex 1\fi=0 
  \usepackage[T1]{fontenc}
  \usepackage[utf8]{inputenc}
  \usepackage{textcomp} 
\else 
  \usepackage{unicode-math}
  \defaultfontfeatures{Scale=MatchLowercase}
  \defaultfontfeatures[\rmfamily]{Ligatures=TeX,Scale=1}
\fi
\IfFileExists{upquote.sty}{\usepackage{upquote}}{}
\IfFileExists{microtype.sty}{
  \usepackage[]{microtype}
  \UseMicrotypeSet[protrusion]{basicmath} 
}{}
\makeatletter
\@ifundefined{KOMAClassName}{
  \IfFileExists{parskip.sty}{%
    \usepackage{parskip}
  }{
    \setlength{\parindent}{0pt}
    \setlength{\parskip}{6pt plus 2pt minus 1pt}}
}{
  \KOMAoptions{parskip=half}}
\makeatother
\usepackage{xcolor}
\IfFileExists{xurl.sty}{\usepackage{xurl}}{} 
\IfFileExists{bookmark.sty}{\usepackage{bookmark}}{\usepackage{hyperref}}
\hypersetup{
  pdftitle={Optimal trading: a model predictive control approach},
  pdfauthor={Simon Clinet; Jean-François Perreton; Serge Reydellet},
  hidelinks,
  pdfcreator={LaTeX via pandoc}}
\usepackage[margin=1in]{geometry}
\usepackage{graphicx}
\makeatletter
\def\maxwidth{\ifdim\Gin@nat@width>\linewidth\linewidth\else\Gin@nat@width\fi}
\def\maxheight{\ifdim\Gin@nat@height>\textheight\textheight\else\Gin@nat@height\fi}
\makeatother
\setkeys{Gin}{width=\maxwidth,height=\maxheight,keepaspectratio}
\makeatletter
\def\fps@figure{htbp}
\makeatother
\providecommand{\tightlist}{%
  \setlength{\itemsep}{0pt}\setlength{\parskip}{0pt}}
\usepackage{amsthm}
\usepackage{algorithm2e}
\ifluatex
  \usepackage{selnolig}  
\fi
\usepackage[]{natbib}
\title{Optimal trading: a model predictive control approach}
\author{Simon Clinet\footnote{Kepler Cheuvreux, Algorithmic Trading
  Quant, PhD,
  \href{mailto:sclinet@keplercheuvreux.com}{\nolinkurl{sclinet@keplercheuvreux.com}},
  +33 1 70 81 57 63} \and Jean-François Perreton\footnote{Kepler
  Cheuvreux, Head of Algorithmic Trading Quant,
  \href{mailto:jperreton@keplercheuvreux.com}{\nolinkurl{jperreton@keplercheuvreux.com}},
  +41 22 994 1306} \and Serge Reydellet\footnote{Kepler Cheuvreux,
  Senior Algorithmic Trading Quant, CFA,
  \href{mailto:sreydellet@keplercheuvreux.com}{\nolinkurl{sreydellet@keplercheuvreux.com}},
  +41 22 994 1334}}
\date{November 03, 2021}

\begin{document}
\maketitle
\begin{abstract}
We develop a dynamic trading strategy in the Linear Quadratic Regulator
(LQR) framework. By including a price mean-reversion signal into the
optimization program, in a trading environment where market impact is
linear and stage costs are quadratic, we obtain an optimal trading curve
that reacts opportunistically to price changes while retaining its
ability to satisfy smooth or hard completion constraints. The optimal
allocation is affine in the spot price and in the number of outstanding
shares at any time, and it can be fully derived iteratively. It is also
aggressive in the money, meaning that it accelerates whenever the price
is favorable, with an intensity that can be calibrated by the
practitioner. Since the LQR may yield locally negative participation
rates (i.e round trip trades) which are often undesirable, we show that
the aforementioned optimization problem can be improved and solved under
positivity constraints following a Model Predictive Control (MPC)
approach. In particular, it is smoother and more consistent with the
completion constraint than putting a hard floor on the participation
rate. We finally examine how the LQR can be simplified in the continuous
trading context, which allows us to derive a closed formula for the
trading curve under further assumptions, and we document a two-step
strategy for the case where trades can also occur in an additional dark
pool.

\par

\end{abstract}

\section{Introduction}

Since the seminal papers of \cite{bertsimas1998optimal} and
\cite{almgren2001optimal}, algorithmic execution has been at the core of
an intense research, and has led to significant improvements along the
way (see e.g \cite{huberman2005optimal}, \cite{gatheral2011optimal},
\cite{obizhaeva2013optimal}, \cite{cartea2015algorithmic},
\cite{gueant2016financial}, and \cite{alfonsi2016dynamic}). A large part
of these contributions has focused on the so-called Implementation
Shortfall (IS) algorithm, whose goal is the minimization under risk
constraints of the expected slippage between the averaged execution
price and the arrival price (i.e the quoted price when the order
starts).

It has been widely acknowledged in the aforementioned works and, for
instance, in \cite{hora2006tactical} and \cite{shen2017hybrid} that
execution strategies should be able to react dynamically to changes in
the trading environment. For instance, it is often required that the
strategy be \textit{aggressive in the money}, meaning that, when buying,
the trading speed should tend to increase whenever the price is low and
decrease when it is high. The rationale behind this behavior is that the
price may feature locally weak to moderate mean-reversion patterns. Such
a behavior may also be motivated by some private information that the
broker has at hand, or by a client's request. For instance,
\cite{gatheral2011optimal} show that their algorithm is aggressive in
the money, and \cite{shen2017hybrid} constructs an algorithm whose
participation rate follows the price curve via squared price slippage
penalization. This is in contrast with the perhaps more standard CARA
framework (see e.g \cite{almgren2001optimal}, \cite{gueant2014vwap},
\cite{gueant2016financial}), which, under an exponential utility
criterion, always yields a deterministic optimal trading curve as shown
in \cite{schied2010optimal}.

In this work, our primary goal is the derivation of a new execution
strategy which benefits from potential price excursions, in a tractable
way. To do so, we set aside the general dynamic programming approach.
Although appealing for the generality of models it can deal with, most
of the time (i) it does not yield closed formulas and (ii) it suffers
from the curse of dimensionality, especially in such complex state
spaces as the ones encountered in algorithmic trading. Instead, we
extend the classical model of \cite{almgren2001optimal} by incorporating
a mean-reversion signal in the price dynamics and show that this new
optimization problem falls under the scope of the celebrated
\textit{linear quadratic regulator} (LQR). The LQR framework
(\cite{bertsekas1995dynamic}) holds whenever the model dynamics can be
described by a linear state transition and is subject to quadratic
costs. It has been applied to several execution problems
(\cite{shen2017hybrid} for the IS case, \cite{busseti2015volume} for the
VWAP case). We take the state process as the couple \((p_t,q_t)\) where
\(p_t = (S_t - \mathfrak S_t )/S_0\) is the price slippage with respect
to some benchmark \(\mathfrak S_t\) and \(q_t \in [0,1]\) is the
proportion of outstanding shares at time \(t\). When
\(\mathfrak S_t \equiv S_0\), we recover an IS strategy, but other
targets can be considered without loss of generality.

One of the main advantages of the LQR compared to more general
stochastic control frameworks is that the optimal participation control
\(u_t^*\) can be explicitly calculated as an affine function of the
state process. Therefore, dynamic features such as aggressiveness in the
money are easily read in the shape of the optimal solution. Our control
problem is quite similar to \cite{shen2017hybrid} who also follows an
LQR approach, although the author explicitly adds a quadratic cost term
in \(\sum_{t}{p_t^2}\) in order to induce reaction to favorable price
movements, whereas our own work directly inputs a mean-reversion
coefficient in the price process used in the optimization procedure. By
doing so, we aim at achieving a more transparent behavior, in line with
the idea that aggressiveness in the money stems from local
mean-reverting excursions. We detail in the next section other
differences between our own model and that of \cite{shen2017hybrid}, in
particular in the way market impact is defined.

As a consequence of the LQR framework, we obtain an optimal strategy
which presents the following advantages: (i) it is affine in the price
process and in the remaining quantity to execute, with coefficients that
can be computed \emph{before the beginning of the execution} following
an iterative procedure; (ii) it is aggressive in the money, with an
intensity that is directly linked to the fictive mean-reversion
parameter; (iii) it is also possible to incorporate into the model other
common features such as an \textit{urgency parameter} and a
\textit{trend signal} without leaving the LQR framework.

An important weakness of the LQR model is the absence of guarantee that
the optimal participation rate will remain positive during the trading
period. However, in most cases, it is highly desirable to avoid
round-trip trading where the algorithm successively buys and sells.
Accordingly, our second goal consists in adapting the above LQR method
in order to circumvent this issue and enforce constraints on the
participation rate. Although one may simply set the participation rate
to zero whenever a negative value is predicted, it is clear that this in
turn may badly affect the trading process. We adopt another strategy and
consider the \textit{certainty equivalent} (CE) model, which amounts to
setting price volatility (the only source of randomness) to zero. Next,
we apply a \textit{model predictive control} (MPC) algorithm
(\cite{garcia1989model}), i.e we see the global costs function as a
quadratic form in the individual participation rates
\((u_t)_{t = 0, \delta t,...,T - \delta t}\) and minimize it for the CE
model, under the constraints \(u_t \geq 0\) for all
\(t = 0, \delta t,...,T - \delta t\). Accordingly, we transpose the
problem to a quadratic optimization with linear constraints, whose
solution can then be computed with standard (and fast) procedures. Any
linear constraint on the participation rates (such as participation
caps) can be harmlessly included in this framework. In theory, running
an MPC model yields an approximation of the optimal participation rate
only, although it is equivalent to the classical LQR in the absence of
constraints (hence it is optimal in that case). Overall, it is fast and
flexible, with very satisfying results from a numerical point of view.

For the sake of completeness, we document the shape of the optimal
strategy in several cases. We first show that the iterative method can
be substantially simplified in the continuous trading limit, that is,
when the rebalancing of the algorithm's inventory can be done
arbitrarily often. In turn, we derive a closed formula for the optimal
strategy when the permanent component of the market impact is
negligible. We finally study the case where orders can be sent to an
additional dark pool. In a model similar to that of
\citep{cartea2015algorithmic}, Section 7.4, we bring forward a two-step
strategy which concurrently sends orders to both venues.

The remainder of the paper is structured as follows. The model in
discrete time is introduced in Section 2. the LQR framework along with
the main results of the paper (Theorem \ref{thmLQR} and Theorem
\ref{thmLQR2}) are introduced in Section 3. Section 4 documents the
constrained optimization through the MPC method (Theorem \ref{thmMpc}
and Proposition \ref{propMPC}). We briefly discuss the calibration of
the mean-reversion signal in Section 5. Section 6 shows how the LQR can
be simplified in the continuous trading limit. Section 7 documents how
the MPC method can be extended to a two-step strategy in the presence of
a dark pool. We conclude in Section 8. Proofs are relegated to the
appendix.

\section{Model}

\subsection*{Quadratic costs}

We introduce our model in discrete time. Let us consider the trading
period \(t=0,\delta t,...,n\delta t = T\), where \(n\) is the number of
trading buckets of the form \([t, t+ \delta t)\) and of length
\(\delta t\). We will denote by \(x_t \delta t\) the quantity (in
shares) which is executed by the algorithm over \([t,t+\delta t)\),
while \(v_t \delta t\) will correspond to the associated expected market
volume over the same period. Note that the actual market volume is, in
general, not known to the practitioner at the beginning of the trading
bucket so that here \(v_t \delta t\) should be understood as an
\textit{estimated} value of the underlying market volume, based on
historical data or any other relevant method. The expected participation
rate is then naturally defined as \(u_t = x_t/v_t\). The
\textit{trading curve} \((u_t)_{t = 0, \delta t, ..., T-\delta t}\) will
be our control process for the remainder of the paper, and we look for
an optimal curve which minimizes a certain cost function in what
follows. It will finally be convenient to introduce\footnote{Note that
  in the expression of \(Q_t\), the sum should be implicitly understood
  as a sum over \(\tau = 0, \delta t, 2\delta t,..., T- \delta t\),
  convention that we adopt for the rest of the paper.}

\begin{itemize}
\item $Q_t = Q_0 - \sum_{\tau = 0}^{t-\delta t} x_\tau \delta t$, the quantity to be executed over the time window $[t,T)$. It is always taken non-negative, whatever the nature of the order (buying/selling). 
\item $q_t = Q_t/Q_0 \in [0,1]$, the proportion of non-executed shares at time $t$.
\end{itemize}

Hereafter, and without loss of generality, we assume that we are trading
on a single stock whose mid-price at time \(t\) is \(S_t\), to which we
associate the \textit{price slippage}
\[p_t = \mathfrak{s} \frac{S_t - \mathfrak S_t}{S_0}\] where
\(\mathfrak{s}\) is the side variable that is \(1\) for a buying order
and \(-1\) for a selling order, and where \(\mathfrak S_t\) is a price
benchmark that the algorithm is supposed to follow as closely as
possible. We retrieve the price slippage for a standard IS algorithm by
setting \(\mathfrak S_t \equiv S_0\), i.e by taking the arrival price as
a reference. In all generality, \(\mathfrak S_t\) may follow other
signals, such as the VWAP, the TWAP, or any other target that the
practitioner finds relevant. For each traded share, the
\textit{execution slippage} is \begin{align} 
\tilde{p}_t = p_t + s_t + \eta_t u_t
\end{align} where \(s_t\) is the signed half-spread in proportion of the
arrival price \(S_0\), and the deterministic quantity \(\eta_t\)
accounts for the temporary impact of trading over the period
\([t, t+\delta t)\). We obtain the corresponding execution slippage
induced by \(x_t\), and in proportion to \(Q_0\) \begin{align}
\tilde{p}_t \frac{x_t}{Q_0} = a_t (p_t+s_t) u_t + a_t \eta_t u_t^2   
\end{align} which is quadratic in \(u_t\) and affine in \(p_t\), where
we have conveniently introduced the scaling quantity \(a_t = v_t/Q_0\)
in a similar fashion as in \cite{shen2017hybrid}.

The primary goal of an IS algorithm is the minimization of the mean
execution slippage, however we also require in general that the strategy
be able to deal with the notions of urgency and completion. This can be
achieved by adding several penalty terms to the global cost function.
Let us introduce the two-dimensional state variable
\(X_t = (p_t,q_t)^T\). In this work, we consider the stage cost at time
\(t < T\) of the form \begin{align}  \label{eqStageCost}
j_t(X_t,u_t)\delta t = \{a_t\eta_t u_t^2 + a_t(p_t+s_t)u_t  + \beta_t q_t^2 \} \delta t,
\end{align}

and the terminal cost at time \(T\) \begin{align}
J_T(X_T) = \tilde{\beta}_T q_T^2.
\end{align}

The global cost function, starting from \(t\) is therefore

\begin{align} \label{eqGlobalCosts}
C_t((X_s)_{s = t,...,T-\delta t}) = \sum_{\tau=t}^{T-\delta t} j_\tau(X_\tau,u_\tau)\delta t + J_T(X_T).
\end{align}

We easily identify the first two terms of the stage cost as the
execution slippage over the bucket \([t, t+\delta t)\). The third term
\(\beta_t q_t^2\) is new, and is an inventory cost which penalizes
strategies holding too many unexecuted shares at time \(t\). The bigger
\(\beta_t\) the stronger the effect, so that \(\beta_t\) acts as an
urgency parameter, that forces the algorithm to select front-loaded
strategies. When \(\beta_t\) is taken proportional to the squared price
volatility, it corresponds exactly to the term coming from the variance
penalty in \cite{almgren2001optimal}. As for the terminal cost, we once
again penalize at the end of the execution any remaining share in the
inventory. This is a smooth way to enforce partial completion, and can
be made hard by taking \(\tilde{\beta}_T = +\infty\), in which case we
are reduced to the Almgren and Chriss framework which assumes the
terminal condition \(q_T = 0\).

Despite its enticing form, the above cost model implicitly assumes that
the spread cost term is proportional to \(s_t u_t\), which is reasonable
only if \(u_t\) is taken non-negative all along the trading curve. A
more realistic formulation would consider instead the cost
\(s_t |u_t|\), which is paid no matter whether the algorithm is buying
or selling at time \(t\). This, in turn, would break the LQR structure
of the problem, which makes this feature undesirable. The problem of
positivity of the participation rate is discussed in Section
\ref{SectionMPC}. For now, we restrict ourselves to the cost model
(\ref{eqGlobalCosts}), and keep in mind that only non-negative trading
curves make sense in this framework.

\subsection*{Price slippage dynamics}

When the price slippage is the sum of an arithmetic Brownian motion and
a linear permanent price impact term, that is a discrete price slippage
of the form
\[ p_{t+\delta t}  = p_t + \mu_t x_t \delta t + \sigma_{t,\delta t} z_t \]
where \(z_t\) are i.i.d standard normal variables, it is well-known that
minimizing (\ref{eqGlobalCosts}) yields a deterministic optimal strategy
(Section 6.5, \cite{cartea2015algorithmic}). Yet, as discussed in the
introduction, we look for trading curves which are aggressive in the
money. \cite{shen2017hybrid} advocates for the addition of another
penalty term in \(a_t\gamma_t p_t^2\) in the stage cost
(\ref{eqStageCost}), and shows that (i) the optimization problem can be
solved by the LQR method and (ii) it induces the desired behavior. In
this work, we take a slightly different route and directly input in the
optimization procedure a modified price dynamics, while leaving the
global cost (\ref{eqGlobalCosts}) as it is. Next, we show that this
approach also falls under the scope of the LQR method and reacts
appropriately to price slippage movements. Specifically, we modify the
above price slippage so that it incorporates some mean-reversion effect
as follows. We take

\[p_{t+\delta t} = (1 - \kappa_t \delta t)p_t + \mu_t x_t \delta t+ \sigma_{t, \delta t} z_t, \]
where, as before, \(\sigma_{t,\delta t}\) is the price slippage
volatility over a bucket, and \(\mu_t\) is the deterministic linear
permanent impact parameter. Note that, in contrast with
\cite{shen2017hybrid}, we have taken the permanent impact proportional
to \(x_t\) and not \(u_t\), because the former is not subject to price
manipulations (\cite{huberman2004price}) whereas the latter is. This is
important for the stability of the strategy as the existence of price
manipulations make strategies containing round-trip trades more
probable, although highly undesirable. The mean-reversion parameter
\(\kappa_t\) will play a crucial role in inducing a price-reaction
behavior in the trading curve. Possible meaning and calibration for
\(\kappa_t\) in practice are discussed in Section \ref{SectionKappa}.

\section{Unconstrained LQR} \label{sectionUnconstrainedLQR}

\subsection*{LQR with mean-reversion signal}

Recall that the state process \(X_t\) is defined as the couple
\((p_t,q_t)^T\). Easy calculations show that \(X_t\) admits a linear
state transition function which, in matrix form, is

\begin{align} \label{eqState}
X_{t+\delta t} = \left( \begin{matrix}  1- \kappa_t \delta t &0  \\ 0 &1 \end{matrix} \right)  X_t + a_t u_t w_t \delta t +  \left( \begin{matrix}  \sigma_{t, \delta t} z_t \\ 0 \end{matrix} \right),
\end{align} where we recall that \(a_t = v_t/Q_0\) and with
\(w_t= \left( Q_0 \mu_t, - 1 \right)^T\). Since the state transition
(\ref{eqState}) is linear in \(X_t\), and the stage and terminal costs
are quadratic in \(X_t\) and \(u_t\), we are under the LQR
specification\footnote{It is actually slightly more general because of
  the presence of the cross term \(a_t p_t u_t\) in the stage cost,
  although the optimization procedure can be performed following the
  same line of reasoning.}. The value function of our optimization
problem is, starting from time \(t\) and state \(X = (p,q)\)

\[ H(t,X) = \inf_{(u_k)_{k = t,...,T-\delta t}} \mathbb{E} \left[\left.\sum_{\tau=t}^{T-\delta t} j_\tau(X_\tau,u_\tau)\delta t + J_T(X_T) \right| X_t=X \right].\]
Key to our analysis is the fact that \(H\) satisfies the Bellman
equation \begin{align} \label{eqBellmanDiscrete}
H(t, X) = \inf_{u \in \mathbb R} \left\{j_t(X, u) \delta t + \mathbb E \left[ H\left(t+ \delta t, \left( \begin{matrix}  1- \kappa_t \delta t &0  \\ 0 &1 \end{matrix} \right) X +  u a_t w_t \delta t+  \left( \begin{matrix} \sigma_{t, \delta t} z_t  \\ 0 \end{matrix} \right) \right)\right] \right \}.
\end{align} Based on (\ref{eqBellmanDiscrete}), we are now ready to
state the main result of this section.

\newtheorem{theorem}{Theorem}
\begin{theorem} \label{thmLQR} (Unconstrained LQR)
Let $X = (p,q)^T$, $\mathcal I = \text{diag}(1,1)$, $\mathcal{J} = \text{diag}(1,0)$, and $e_1 = (1,0)^T$. The value function $H$ is quadratic in $X$, of the form $H(t,X) = X^T P_t X + b_t^T X + e_t$. The optimal control $u_t^*$ is of the form 
\begin{align}  \label{eqU}
u_t^* = k_t^T X_t + f_t 
\end{align}
with 
$$k_t = -\frac{1}{g_t}\left((\mathcal I - \kappa_t \delta t \mathcal J)P_{t+\delta t}^Tw_t + \frac{1}{2}e_1\right),$$
$$ f_t = -\frac{b_{t+\delta t}^T w_t + s_t}{2g_t},$$
$$ g_t = \eta_t + a_t w_t^T P_{t+\delta t} w_t \delta t .$$
Moreover, we have the backward iteration scheme 
$$ P_t =  P_{t+\delta t} - \kappa_t \delta t (\mathcal J P_{t+\delta t} + P_{t+\delta t} \mathcal J) + \kappa_t^2 (\delta t)^2 \mathcal J P_{t+\delta t} \mathcal J + \left[\Delta_t  - a_tg_tk_t k_t^T\right]\delta t,$$
$$ b_t = b_{t+\delta t} -\left\{2a_tg_tf_tk_t + \kappa_t \mathcal Jb_{t+\delta t}   \right\}\delta t,$$
$$e_t = e_{t+\delta t} + \sigma_{t,\delta t}^2 e_1^T P_{t+\delta t}e_1     - a_t g_t f_t^2  \delta t $$
where $\Delta_t = \left( \begin{matrix}  0 &0  \\ 0 & \beta_t \end{matrix} \right)$, and with the terminal conditions $ P_T = \left( \begin{matrix} 0 &0  \\ 0 & \tilde{\beta}_T \end{matrix} \right)$, $b_T = (0,0)^T$, $e_T=0$. 
\end{theorem}

We now give several important remarks.

\begin{itemize}
\item We immediately read from  (\ref{eqU}) that $u_t^*$ is affine in the price, and will react linearly with respect to deviations from the benchmark price. Although less clear on the formula, our numerical results show that, when buying (resp. selling), it is \textit{negatively} (resp. \textit{positively}) related to $p_t$ so that it \textit{does} benefit from price excursions (and is aggressive in the money). Note also that $u_t^*$ is also affine in the remaining quantity, which is a direct consequence of the presence of the completion and urgency parameters.
\item Formula (\ref{eqU}) for $u_t^*$ is not closed, since it can be obtained only through an iterative scheme. However, it is exact since there is no approximation in the derivation of the induction. More importantly, the quantities $P_t$, $b_t$, $e_t$ and therefore $k_t$, $f_t$, and $g_t$ can all be computed \textit{before} the beginning of the execution, at time $t=0$, and \textit{only once}. There is no need to recalculate these quantities at each time step. A detailed pseudo-code is given in Algorithm \ref{algoLQR} below.
\item It is common, as in \cite{almgren2001optimal}, to take the urgency parameters proportional to $\sigma_{t, \delta t}^2$. By doing so, we make sure that increased price risk yields more front-loaded strategies.  
\item When $\kappa_t = 0$ (and $\tilde{\beta}_T = +\infty$), the problem corresponds to the classical Almgren and Chriss framework. In that case, we have a closed formula for $u_t^*$ which depends on $q_t$ only and not on $p_t$. When $\kappa_t \neq 0$, solving analytically the above scheme in its full generality seems difficult, which we leave for further research, although we document in Section \ref{SectionContinuous} the analytical solution in the continuous trading framework with no permanent impact ($\mu_t=0$). 
\item An important feature of the above strategy is that the well-posedness of the iteration scheme may break if $g_t$ vanishes, which can happen if $P_t$ is not symmetric and positive definite. While the symmetry of $P_t$ can be proved by induction in a straightforward manner, it is not clear whether positive definiteness holds in general. However, numerical simulations suggest that it does hold in practice for reasonable choices of impact parameters.   
\item In the limit $\delta t \to 0$, we approach the continuous trading framework, and the above formulas can be substantially simplified. In particular, note that in the continuous limit we have $g_t = \eta_t$ so that it can never vanish if temporary impact is present. We have documented this case in Section \ref{SectionContinuous}. 
\end{itemize}

Figure 1 shows an example of an execution with \(\kappa_t > 0\),
executed with an actual price that is not mean-reverting. We have also
represented the Almgren and Chriss case (\(\kappa_t =0\)). Market
volumes are flat, there is a mild urgency level during the algorithm
with the hard completion constraint \(q_T =0\). It is clear that the
trading curve dynamically reacts to price movements, except near the end
of the execution where the completion condition forces the algorithm to
accelerate no matter the value of the price slippage.

\begin{figure}
\centering
\includegraphics{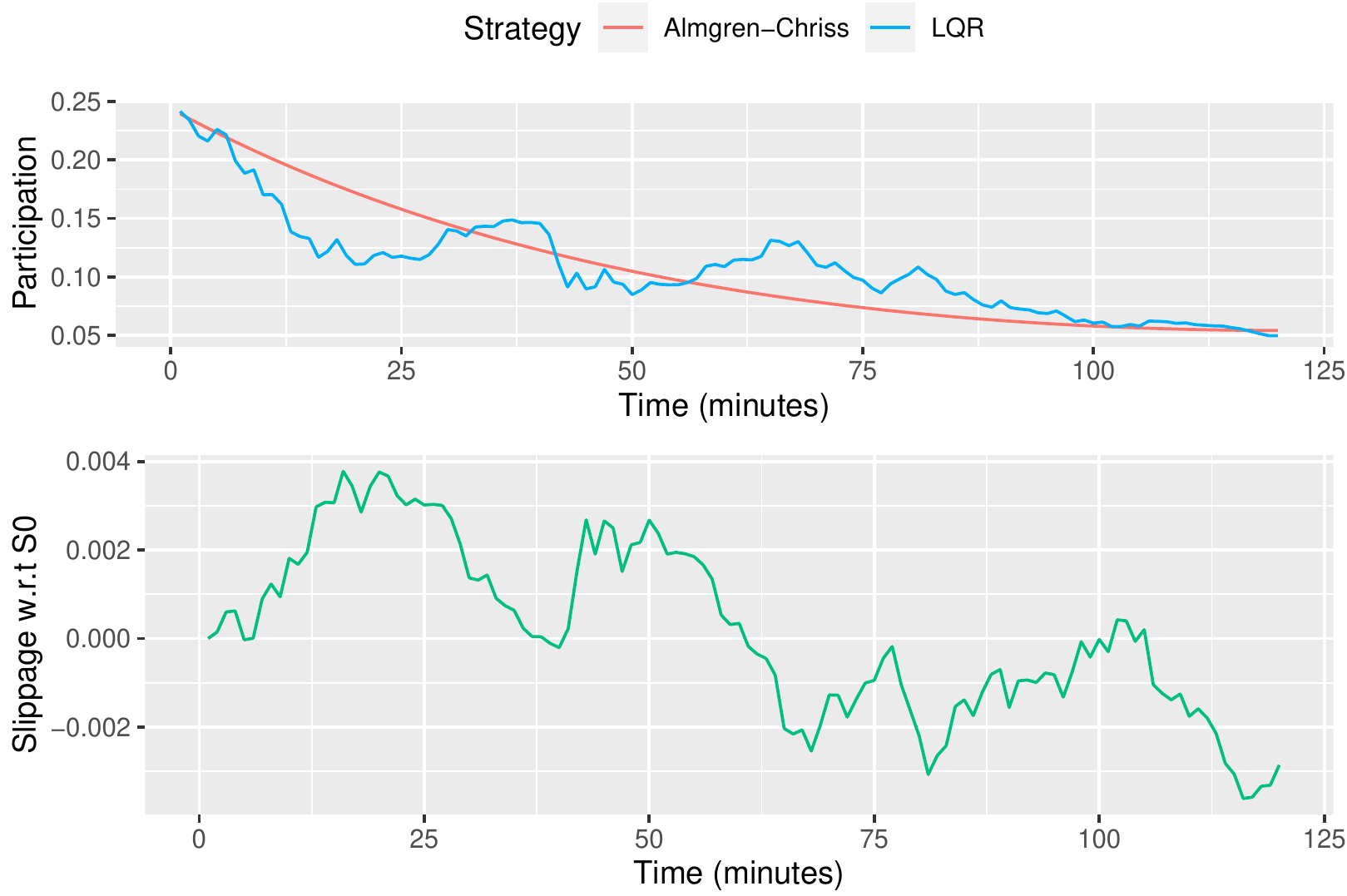}
\caption{Unconstrained LQR strategy execution}
\end{figure}

\begin{algorithm}[H] \label{algoLQR}
 \KwIn{$(a_t,\eta_t,w_t,\Delta_t,\kappa_t,s_t)_{t =0,\cdots, T-\delta t}$.}
 \KwOut{The optimal LQR trading curve $u_0^*$, $u_{\delta t}^*, \cdots, u_{T-\delta t}^*$.}
  Initialize clock time $\tau \leftarrow 0$ \;
  $ P_T \leftarrow \left( \begin{matrix} 0 &0  \\ 0 & \tilde{\beta}_T \end{matrix} \right)$, $b_T \leftarrow (0,0)^T$ \;
  \For{$t=T - \delta t$ down to $0$}{
  $ g_t \leftarrow \eta_t + a_t w_t^T P_{t+\delta t} w_t \delta t $ \;
  $k_t \leftarrow -\frac{1}{g_t}\left((\mathcal I - \kappa_t \delta t \mathcal J)P_{t+\delta t}^Tw_t + \frac{1}{2}e_1\right)$ \;
$ f_t \leftarrow -\frac{b_{t+\delta t}^T w_t + s_t}{2g_t}$ \;
  \If{t > 0}{
    $ P_t \leftarrow  P_{t+\delta t} - \kappa_t \delta t (\mathcal J P_{t+\delta t} + P_{t+\delta t} \mathcal J) + \kappa_t^2 (\delta t)^2 \mathcal J P_{t+\delta t} \mathcal J + \left[\Delta_t  - a_tg_tk_t k_t^T\right]\delta t$ \;
    $ b_t \leftarrow b_{t+\delta t} -\left\{2a_tg_tf_tk_t + \kappa_t \mathcal Jb_{t+\delta t}   \right\}\delta t$ \;
  }
  }
  \For{$t = 0$ to $T-\delta t$}{
  Wait for clock time $\tau \geq t$\;
  Fetch $Q_t$, $\mathfrak{S}_t$, $S_t$\;
  $X_t \leftarrow \left(  \frac{S_t - \mathfrak{S}_t }{S_0}, \frac{Q_t}{Q_0}\right)^T$\;
  $u_t^* \leftarrow k_t^T X_t + f_t$\;
  }
 \caption{Online calculation for the LQR optimal policy $(u_t^*)_{0 \leq t \leq T - \delta t}$.}
\end{algorithm}

\subsection*{Trading with an additional trend signal}

For the sake of completeness and since it is harmless for the
optimization procedure, we now examine the slightly more general price
slippage model \begin{align}
p_{t+\delta t} = (1 - \kappa_t \delta t)p_t + \mu_t x_t \delta t + \alpha_t \delta t  + \sigma_{t, \delta t} z_t, 
\end{align} where \(\alpha_t\) is a trend signal which impacts the
direction in which the price slippage is drifting. It encompasses any
information that may be available to the investor who anticipates a
shift in the price over the period \([0,T]\). The new state transition
equation can be written as

\begin{align} \label{eqState2}
X_{t+\delta t} = \left( \begin{matrix}  1- \kappa_t \delta t &0  \\ 0 &1 \end{matrix} \right)  X_t + a_t u_t w_t \delta t +  \left( \begin{matrix} \alpha_t \delta t + \sigma_{t, \delta t} z_t \\ 0 \end{matrix} \right).
\end{align} The LQR can be adapted and yields the following result.

\begin{theorem} \label{thmLQR2} (Unconstrained LQR with a trend signal)
Let $X = (p,q)^T$, $\mathcal I = \text{diag}(1,1)$, $\mathcal{J} = \text{diag}(1,0)$, and $e_1 = (1,0)^T$. The value function $H$ is quadratic in $X$, of the form $H(t,X) = X^T P_t X + b_t^T X + e_t$. Moreover, the optimal control $u_t^*$ is of the form 
$$ u_t^* = k_t^T X_t + f_t $$
with 
$$k_t = -\frac{1}{g_t}\left((\mathcal I - \kappa_t \delta t \mathcal J)P_{t+\delta t}^Tw_t + \frac{1}{2}e_1\right),$$
$$ f_t = -\frac{\left[b_{t+\delta t}^T + 2\alpha_t \delta te_1^TP_{t+\delta t}\right]w_t + s_t}{2g_t},$$
$$ g_t = \eta_t + a_t w_t^T P_{t+\delta t} w_t \delta t .$$
Moreover, we have the backward iteration scheme 
$$ P_t =  P_{t+\delta t} - \kappa_t \delta t (\mathcal J P_{t+\delta t} + P_{t+\delta t} \mathcal J) + \kappa_t^2 (\delta t)^2 \mathcal J P_{t+\delta t} \mathcal J + \left[\Delta_t  - a_tg_tk_t k_t^T\right]\delta t,$$
$$ b_t = b_{t+\delta t} -\left\{2a_tg_tf_tk_t + \kappa_t \mathcal Jb_{t+\delta t}  - 2\alpha_t  (\mathcal I - \kappa_t \delta t\mathcal J) P_{t+\delta t} e_1  \right\}\delta t,$$
$$e_t = e_{t+\delta t} +  [\sigma_{t,\delta t}^2+ \alpha_t^2(\delta t)^2] e_1^T P_{t+\delta t}e_1 + \left[  \alpha_t  b_{t+\delta t}^Te_1   - a_t g_t f_t^2 \right]  \delta t $$
where $\Delta_t = \left( \begin{matrix}  0 &0  \\ 0 & \beta_t \end{matrix} \right)$, and with the terminal conditions $ P_T = \left( \begin{matrix} 0 &0  \\ 0 & \tilde{\beta}_T \end{matrix} \right)$, $b_T = (0,0)^T$, $e_T=0$. 
\end{theorem}

\subsection*{Targeting the time-varying market VWAP}

It is interesting to consider what happens when the benchmark price is
set as the dynamic market VWAP (excluding the algorithm's own
contribution for simplicity),
\(\mathfrak S_t = V_t^{-1}\sum_{\tau = 0}^{t} v_\tau S_\tau \delta t\)
where \(V_t = \sum_{\tau =0}^{t} v_\tau \delta t\). The TWAP case is
obtained, as usual, by setting \(v_\tau \equiv 1\). In that case,
straightforward calculations give
\[ p_{t+\delta t} = \left(1 - \frac{v_{t+\delta t}}{V_{t+\delta t}} \delta t \right) p_t + \left(1+\frac{v_{t+\delta t}}{V_{t+\delta t}} \delta t\right)\frac{S_{t+\delta t} - S_t}{S_0},\]
so that independently of the dynamics of the price process \(S_t\), the
slippage already presents mean-reversion. In that case, and in the
absence of any other information, an option consists in setting
\(\kappa_t = a\frac{v_t}{V_{t+\delta t}}\) where \(v\) and \(V\) are the
expected volumes introduced in the model section, and \(a\) is a
parameter controlling the intensity of the signal. Here, the algorithm
will be again aggressive in the money, in the sense that it will
accelerate when the spot price \(S_t\) is above the market VWAP. Note
however that the mean-reversion parameter \(\kappa_t\) is strong at
first, but then asymptotically vanishes toward the end of long
executions (indeed, in general we have
\(\frac{ v_{t+\delta t}}{ V_{t+\delta t}} \sim \frac{C}{t} \to 0\) when
\(t \to T \gg 1\)) so that the trading curve will loose its
price-reactiveness over time. Interestingly, even small values for \(a\)
(see below) already yield quite reactive strategies for standard market
impact parameters.

Figure 2 shows an example of an execution with benchmark
\(\mathfrak S_t = VWAP_t\), \(\kappa_t\) as above and \(a = 0.015\). As
before, we have also represented the Almgren and Chriss case
(\(\kappa_t = 0\)), market volumes are flat, there is a mild urgency
level during the algorithm with the hard completion constraint
\(q_T =0\).

\begin{figure}
\centering
\includegraphics{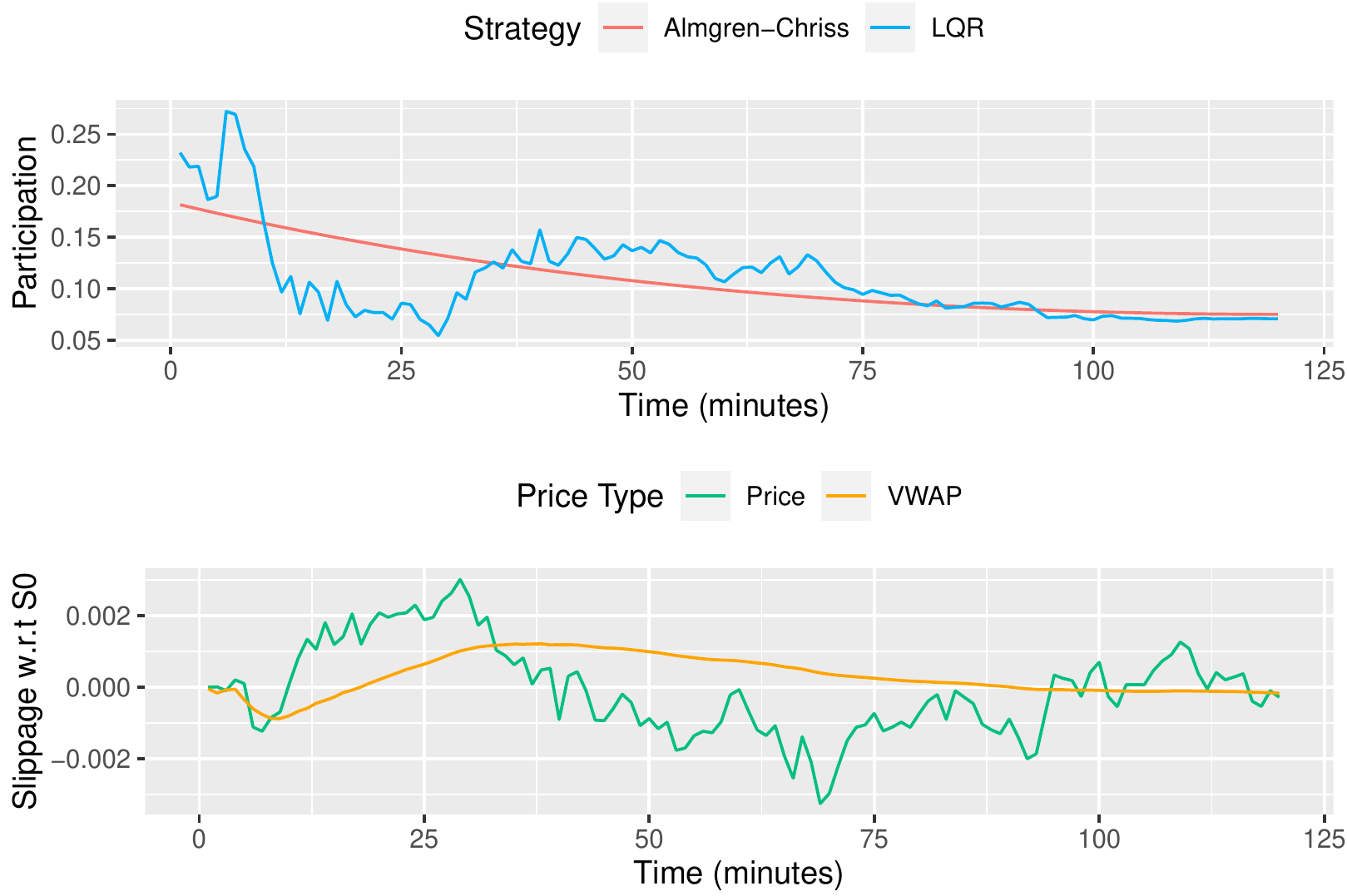}
\caption{Unconstrained LQR strategy execution when targeting the
time-varying market VWAP}
\end{figure}

\section{Constrained LQR via MPC} \label{SectionMPC}

As explained in the introduction, the LQR from the previous section
presents one major drawback:
\textit{there is no guarantee that $u_t^*$ will remain positive during the trading period}.
In fact, the larger the mean-reversion signal, the more the strategy
deviates from its Almgren-Chriss equivalent \(\kappa_t =0\) and the more
likely are negative values for \(u_t^*\). This feature is, in general,
to be avoided because (i) from a practical point of view, round trip
trades are unwanted by most traders and (ii) as pointed out in the
previous section, our global cost function (\ref{eqGlobalCosts})
implicitly assumes \(u_t\) to be non-negative through the presence of
the linear spread term \(s_tu_t\), which should be \(s_t|u_t|\) in all
generality, thus breaking the LQR structure of the problem. The simplest
remedy against such a feature is a hard floor on the rate (this is, for
instance, suggested in \cite{busseti2015volume}, p.13)
\[ \tilde{u}_t^* = \max(k_t^T X_t + f_t, 0).\] However it is not
satisfactory for two reasons. First, since \(\tilde{u}_t^* \geq u_t^*\),
the strategy based on \(\tilde{u}_t^*\) will tend to complete the order
too early and will not benefit from late opportunities that may present
themselves around the horizon time \(T\). Next, if at a given time
\(t=t_0\) the threshold condition is met and \(\tilde{u}_{t_0}^* = 0\),
this fact \textit{should impact} the subsequent participation rates
\(\tilde{u}_{t_0 + \delta t}^*, ..., \tilde{u}_{T - \delta t}^*\),
which, with the above rule, would not be the case.

In this work, we set aside the hard threshold method and bring forward a
more flexible solution based on the celebrated Model Predictive Control
(MPC) framework. The MPC method, sometimes also called Receding Horizon
Control (RHC) in the literature (see \cite{bemporad2002explicit},
\cite{bemporad2003suboptimal}), has a long history in control theory
which can be traced back to the late 1970s (\cite{richalet1978model},
\cite{cutler1980dynamic}, \cite{garcia1989model},
\cite{morari1999model}), with a tremendous number of applications in the
industry, ranging from temperature control in chemistry
(\cite{martin1986predictive}) to missile guidance
(\cite{li2014missile}). Table 6 in \cite{qin2003survey} reports more
than 2,000 applications in the refining, chemicals, petrochemicals,
mining, defense, and automotive industries to name a few. The model owes
its popularity to its ability to deal with quadratic controls featuring
linear constraints in a relatively fast way following a quadratic
programming algorithm. Although the MPC method is heuristic, it is
optimal in the absence of constraints (see Proposition \ref{propMPC}),
and otherwise it remains numerically close to the LQR strategy while
avoiding trading curves that break these constraints (see Figure 3 for a
visual example). In other words, the MPC method is equivalent to the LQR
of Theorem \ref{thmLQR} when the variables are free of constraints
(although in that case, it is recommended to use the iterative scheme in
Theorem \ref{thmLQR} for obvious computation costs reasons).

the MPC methodology consists in finding an optimal
\textit{deterministic} curve (that is, conditionally to the information
available at the current bucket) over a time window spanning part or all
of the remaining future periods, and then keeping only its first element
as our policy for the upcoming bucket. At the beginning of the next
period, a new deterministic curve is calculated, and the procedure keeps
going until the horizon time \(T\) is reached. To obtain such a
deterministic behavior, the state process is assumed to follow a
Certainty Equivalent (CE) dynamics, where all sources of randomness are
cancelled. Although such hypothesis may appear extremely unrealistic at
first glance, it turns out, as explained below, that this has actually
little to no impact on the problem at hand while greatly simplifying the
analysis. In our optimal trading framework, the method goes as follows.
First, the CE dynamics corresponds to setting
\(\sigma_{t, \delta t} = 0\) in the state process dynamics. This yields
the state evolution equation
\[  X_{t+\delta t} = \left( \begin{matrix}  1- \kappa_t \delta t &0  \\ 0 &1 \end{matrix} \right)  X_t + a_t u_t w_t \delta t + \alpha_t e_1 \delta t.\]
Next, when at a given time \(t\), denote the expected cost as a
quadratic function in \(u_t, u_{t+\delta t}, ..., u_{T-\delta t}\), by
\(C_t(u_t, u_{t+\delta t}, ..., u_{T-\delta t})\). We give in Theorem
\ref{thmMpc} below the exact shape of \(C_t\). At this point, define
\begin{eqnarray}\label{eqMPC}
(\bar{u}_{t},\bar{u}_{t+\delta t}...,\bar{u}_{T-\delta t}) = \text{argmin}_{u_t \geq 0,...,u_{T-\delta t} \geq 0}  C_t(u_t, u_{t+\delta t}, ..., u_{T-\delta t})
\end{eqnarray} which is the solution of a quadratic optimization under
linear constraints, that can be calculated following fast and standard
quadratic programming methods. Finally, take
\[ \tilde{u}_t^* = \bar{u_t},\] and discard
\(\bar{u}_{t+\delta t}...,\bar{u}_{T-\delta t}\). At time
\(t + \delta t\), start over the optimization procedure in
\((u_{t+\delta t}, u_{t+2\delta t}...,u_{T-\delta t})\), calculate
\(\tilde{u}_{t+\delta t}^*\), and so on. Algorithm \ref{algoMPC}
documents the above calculations in details. The MPC approach is
appealing for several reasons. First, since it re-computes the
certainty-equivalent optimal solution at each time step, the fact that
\(\tilde{u}_{t_0}^*\) reaches the constraint \emph{does} impact the
shape of the whole remaining trading curve
\(\tilde{u}_{t_0 + \delta t}^*, ..., \tilde{u}_{T - \delta t}^*\).
Moreover, more general constraints can be incorporated in the
optimization, as long as they remain linear in the control participation
rates. For instance, participation rates may be individually capped
(\(u_t \leq u_{max}\)), or constraints could be formulated on their
running means (\(k^{-1}(u_t+...+u_{t+k \delta t}) \geq u_{min}\) ), and
so on. We are now ready to state our main results.

\begin{theorem} \label{thmMpc}
For $s \leq t$, let $u_{t:(T-\delta t)} = (u_t, u_{t+\delta t}, ..., u_{T-\delta t})^T \in \mathbb{R}^k$ for some $k \leq n$, and $\pi_{s:t}^\kappa = \prod_{\theta = s}^t (1-\kappa_\theta \delta t)$. For $t \leq \tau$, 
\begin{eqnarray} \label{eqL}
L_\tau = \left(\begin{matrix} \pi_{(t+\delta t):(\tau-\delta t)}^\kappa v_t\mu_{t}\delta t & \pi_{(t+2\delta t):(\tau-\delta t)}^\kappa v_{t+\delta t}\mu_{t+\delta t}\delta t &...&  v_{\tau-\delta t}\mu_{\tau-\delta t}\delta t & 0 &0& ...\\  -a_{t}\delta t & -a_{t+\delta t}\delta t& ... &   -a_{\tau-\delta t}\delta t&0 & 0&... \\0&0&...&0&1&0&...\end{matrix}\right) \in \mathbb R^{3 \times k},
\end{eqnarray}
and
\begin{eqnarray} \label{eqY}
Y_\tau = \left(\begin{matrix}   \pi_{t:(\tau-\delta t)}^\kappa p_{t} + \sum_{\theta = t}^{\tau-\delta t} \pi_{\theta:(\tau-\delta t)}^\kappa \alpha_\theta \delta t\\q_{t} \\0\end{matrix} \right).
\end{eqnarray}
Then we have the representation 
\begin{eqnarray*} 
C_t(u_t, u_{t+\delta t}, ..., u_{T-\delta t}) &=& u_{t:(T-\delta t)}^T \mathbf{A}_t u_{t:(T-\delta t)}+\mathbf{b}_t^T u_{t:(T-\delta t)}  + \mathbf{c}_t
\end{eqnarray*}
where
\begin{eqnarray*}
\mathbf{A}_t &=& \left\{\sum_{\tau = t}^{T-\delta t} L_\tau^T M_\tau L_\tau \right\}\delta t + L_T^T M_T L_T, \\ \mathbf{b}_t^T &=& \left\{\sum_{\tau = t}^{T-\delta t} 2 Y_{\tau}^T M_\tau L_\tau + N_\tau^T L_\tau \right\} \delta t + 2 Y_{T}^T M_T L_T + N_T^T L_T,
\end{eqnarray*}
$\mathbf{c}_t$ is a constant term, and where for $t< T$ we have $$M_t = \left(\begin{matrix} 0 &0&a_t/2\\0&  \beta_t& 0 \\a_t/2&0&a_t\eta_t\end{matrix}\right), \text{ and } N_t =   \left(\begin{matrix}0\\0\\ a_ts_t\end{matrix}\right)$$
and at time $T$ 
$$M_T = \left(\begin{matrix} 0 &0&0\\0&  \tilde{\beta}_T& 0 \\0&0&0\end{matrix}\right), \text{ and } N_t =   \left(\begin{matrix}0\\0\\0\end{matrix}\right).$$
\end{theorem}

The next proposition shows that MPC and LQR are equivalent in the
absence of constraints. The main reason behind this result is that, all
things being equal, the LQR coefficients are independent from the
volatility \(\sigma_{t, \delta t}\) which is the unique source of
randomness in the state evolution equation. Indeed, remark that only the
constant \(e_t\) depends on volatility in Theorem \ref{thmLQR}, which
plays no role in the optimization. In other words, the LQR strategy is
the solution to the certainty equivalent problem with no constraints.
This also suggests that assuming the CE dynamics for the MPC approach
has little impact on the numerical solution.

\newtheorem{proposition}{Proposition} 
\begin{proposition} \label{propMPC}
In the absence of constraints, the optimal solutions respectively obtained by MPC and by LQR coincide.
\end{proposition}

\begin{figure}
\centering
\includegraphics{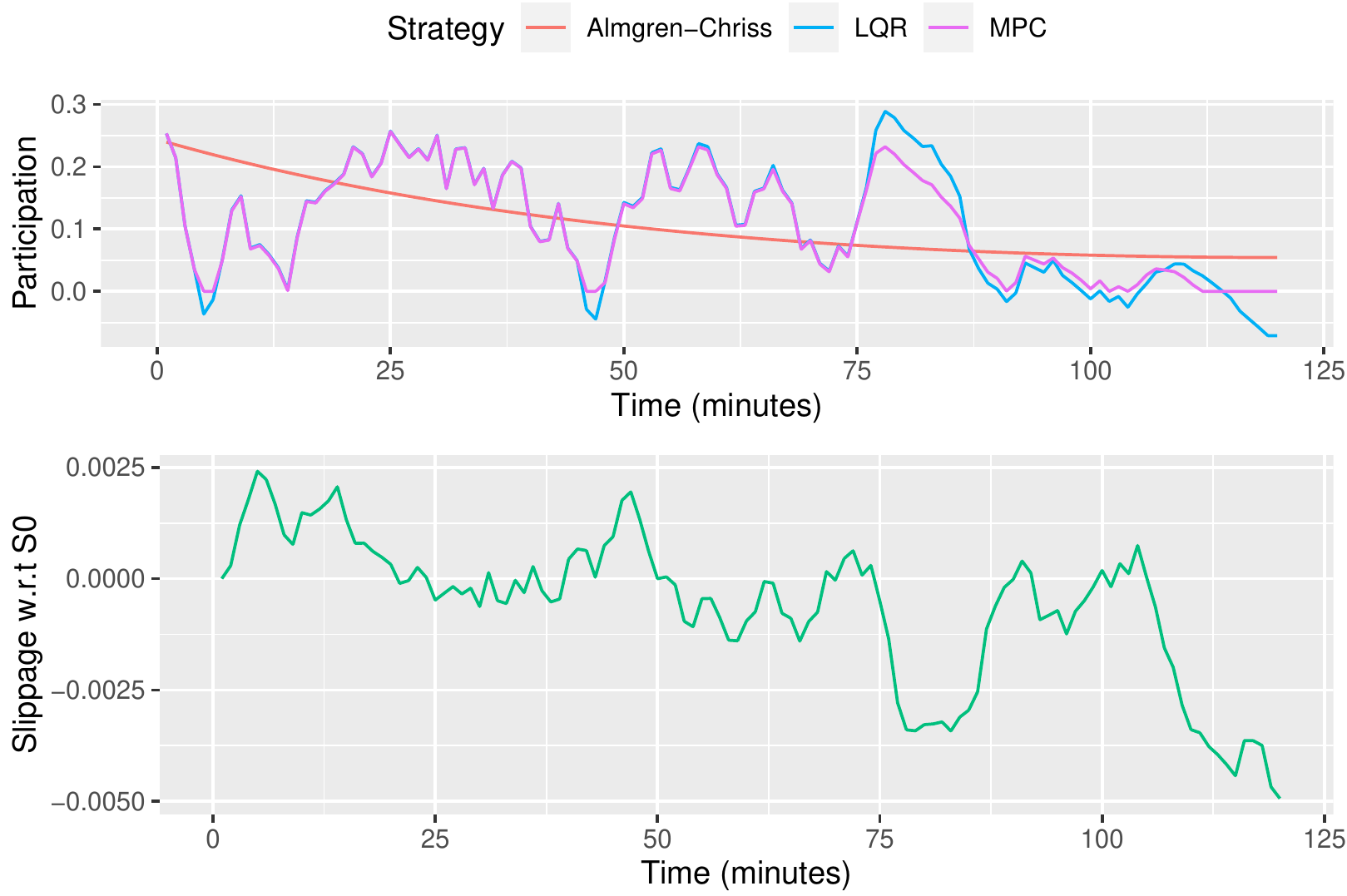}
\caption{MPC strategy execution}
\end{figure}

\begin{algorithm}[H]\label{algoMPC}
 \KwIn{$(M_t, N_t)_{t =0,\cdots, T}, (a_t, v_t, \mu_t, \eta_t, \alpha_t, \kappa_t)_{t =0,\cdots, T-\delta t}$.}
 \KwOut{The optimal MPC trading curve $u_0^*$, $u_{\delta t}^*, \cdots, u_{T-\delta t}^*$.}
  Initialize clock time $\tau \leftarrow 0$ \;
  \For{$t = 0$ to $T-\delta t$}{
    Wait for clock time $\tau \geq t$\;
    Fetch $Q_t$, $\mathfrak{S}_t$, $S_t$\;
    $p_t \leftarrow  \frac{S_t - \mathfrak{S}_t }{S_0}, q_t \leftarrow\frac{Q_t}{Q_0}$\; 
  \For{$\tau = t$ to $T-\delta t$}{
    $L_\tau \leftarrow$ (\ref{eqL})\;
    $Y_\tau \leftarrow$ (\ref{eqY})\;
  }
  $\mathbf{A}_t \leftarrow \left\{\sum_{\tau = t}^{T-\delta t} L_\tau^T M_\tau L_\tau \right\}\delta t + L_T^T M_T L_T$ \;
    $\mathbf{b}_t^T \leftarrow \left\{\sum_{\tau = t}^{T-\delta t} 2 Y_{\tau}^T M_\tau L_\tau + N_\tau L_\tau \right\} \delta t + 2 Y_{T} M_T L_T + N_T L_T$\;
    $(\bar{u}_{t},\bar{u}_{t+\delta t}...,\bar{u}_{T-\delta t}) \leftarrow \text{argmin}_{u_t \geq 0,...,u_{T-\delta t} \geq 0}u_{t:(T-\delta t)}^T \mathbf{A}_t u_{t:(T-\delta t)}+\mathbf{b}_t^T u_{t:(T-\delta t)}$\;
    $u_t^* \leftarrow \bar{u}_{t}$\;
  }
 \caption{Online calculation for the MPC optimal policy $(u_t^*)_{0 \leq t \leq T - \delta t}$.}
\end{algorithm}

Figure 3 shows an example of an MPC execution, with a strong signal
\(\kappa_t > 0\), executed with an actual price that is not
mean-reverting. We have also represented the Almgren and Chriss case
(\(\kappa_t =0\)), and the unconstrained LQR case. Market volumes are
flat, there is a mild urgency during the algorithm with the hard
completion constraint \(q_T =0\). Notice how the MPC curve remains very
close to the LQR strategy while satisfying the positivity constraints.

We finally examine the practical implementation issues that may arise
wen applying the above MPC approach for long orders. It is indeed
important to note that when the number of buckets \(n\) is large,
finding the value of the quadratic program (\ref{eqMPC}) may be
challenging, in particular if it is done in an online manner during the
meta-order's lifespan\footnote{For an order following 5-minute buckets
  spanning the whole trading day, this gives \(n=102\) constraints at
  the beginning of the order for most European markets, which remains
  feasible with reasonable computing power, and is satisfying for most
  practical cases. However, for shorter buckets, the problem quickly
  becomes very involved.}. We suggest two potential solutions to this
issue, although their practical implementation is beyond the scope of
this paper and will be examined in future research.

\begin{itemize}
\item Following the original spirit of Receding Horizon Control, one option consists in reducing the time horizon in the optimization (\ref{eqMPC}). Specifically, we may define $K_{RHC}$ as the maximum number of buckets the scheduler looks ahead, and rewrite for each time $t$ the quadratic problem over the \textit{rolling} period $t, t+\delta t, ..., \tau^K_{RHC}$ where $\tau_{RHC}^K = \min (t+ K_{RHC} \delta t, T-\delta t)$.
\begin{eqnarray*}
(\bar{u}_{t},\bar{u}_{t+\delta t}...,\bar{u}_{\tau^K_{RHC}}) = \text{argmin}_{u_t \geq 0,...,u_{\tau^K_{RHC}} \geq 0}  C_t(u_t, u_{t+\delta t}, ..., u_{\tau^K_{RHC}}).
\end{eqnarray*}
$K_{RHC}$ can be then calibrated to find a balance between computational feasibility and prediction power. This method is easily implemented and can tremendously reduce the complexity of the problem, although it makes assessing the quality of the predicted value $\bar{u}_t$ difficult in practice. In particular, note that any optimization at time $t < T - (K_{RHC} + 1)\delta t$ will ignore the terminal completion constraint $\tilde{\beta}_T q_T^2$, which may be potentially harmful for the shape of the trading curve. This can be partially solved by taking inventory constraints $\beta_t$ that are strongly increasing in $t$.   
\item As suggested in, e.g \cite{alessio2009survey}, it is also possible to run an \textit{offline} optimization procedure in a variety of cases and tabulate the results. The authors show that the above quadratic program always yields solutions that are affine in the state process as in the unconstrained LQR case, but where now the coefficients in the affine representation depend on whether the state process lies in particular critical hyperplanes or not. Ideally, the authors suggest that the hyperplanes equations and the associated affine coefficients be computed and stored prior to the beginning of the execution. However, this method is challenging in its own way because the number of such critical regions grows exponentially with the number of constraints (i.e, in our case, buckets). A fast, heuristic implementation is proposed in Section 3.1 of \cite{alessio2009survey}.       
\end{itemize}

\section{Calibrating $\kappa_t$} \label{SectionKappa}

Although stock mid-prices may present weak to moderate mean-reversion
patterns locally in time, the intensity of such phenomena is in general
unstable and difficult to capture from a statistical point of view. This
is further complicated when using high frequency historical data
contaminated by the so-called microstructure noise. Accordingly, in what
follows, we propose to see \(\kappa_t\) as a tuning parameter for
controlling the deviation of the trading curve from its Almgren-Chriss
curve counterpart. Put differently, \(\kappa_t\) can be calibrated so
that our algorithm reaches a given level of aggressiveness when it is in
the money.

Let us consider the unconstrained LQR case for the sake of simplicity,
with a mean-reversion signal that is fixed in time
\(\kappa_t = \kappa\). Recall that, by definition, when \(\kappa\) is
set to \(0\) in the optimization procedure, the resulting trading curve
corresponds to the deterministic strategy of \cite{almgren2001optimal}.
When \(\kappa >0\), we allow the algorithm to deviate and wander around
this curve with an amplitude that increases with \(\kappa\). The
mean-reversion signal can thus be calibrated by setting a desired
deviation level and taking the corresponding \(\kappa\). For a given
order execution, consider \(u_{LQR,t}\) and \(u_{AG,t}\) the trading
curves respectively obtained from the LQR method with a given
\(\kappa > 0\) and the one obtained when setting \(\kappa = 0\). Let us
set \(q_a\) the \(a\)-quantile of
\(\{ u_{LQR,t}-u_{AG,t}, t \in [0,T-\delta t]\}\), and set
\[ \delta_a = \frac{q_a + q_{1-a}}{2}.\] We find empirically that, for
reasonable ranges of parameters, \(\delta_a\) is mostly explained by the
log linear relation
\[ \log \delta_a = b_0 + b_1\log \kappa + b_2 \log \nu + b_3 \log \frac{Q_0}{M}\]
where \(M = \sum_{\tau = 0}^{T-\delta t}{v_\tau}\) is the expect market
volume over the execution period, and \(\nu\) is a market impact
parameter such that \(\mu_t\) and \(\eta_t\) are taken of the form
\(\mu_t = \mu_{0,t} \nu\) and \(\eta_t = \eta_{0,t} \nu\). By setting
\(\delta_a\) to a desired deviation (i.e aggressiveness) level, one can
invert the above relation and get
\[\log \kappa = \frac{1}{b_1}\log \delta_a - \frac{b_0}{b_1} - \frac{b_2}{b_1} \log \nu  -\frac{b_3}{b_1} \log \frac{Q_0}{M}.\]
The coefficients \(b_0,...,b_3\) can be estimated on a mix of simulated
and historical data. Then, selecting a small value for \(\delta_a\) will
yield a conservative execution style close to the Almgren-Chriss
framework, whereas taking larger values will let the trading curve
modulate its velocity following the price slippage \(p_t\). In any case,
the presence of \(\delta_a\) ensures that with high probability, the
trading curve excursion will remain bounded within a compact interval,
to avoid uncontrolled spikes in aggressiveness.

\section{The continuous limit} \label{SectionContinuous}

In this section, we return to the LQR framework and we document the
optimization problem from Section \ref{sectionUnconstrainedLQR} in a
continuous time setting. We show that the expression of the optimal
strategy can be substantially simplified in the continuous trading limit
(Theorem \ref{thmLQRContinuous}), and even allows us to derive a closed
formula in the particular case of no permanent impact (\(\mu_t=0\)) and
constant parameters (Theorem \ref{thmFormula} and Corollary
\ref{CorFormula}).

There are two ways to obtain the resulting optimal strategy and optimal
value function in this framework. They can be informally guessed by
taking all the expressions in Theorem \ref{thmLQR} and then taking the
limit \(\delta t \to 0\) with \(T\) fixed. Alternatively, we can also
explicitly derive the Hamilton-Jacobi-Bellman (HJB) equation associated
to the quadratic problem, and apply a verification theorem (see, e.g
Theorem 5.2, \cite{cartea2015algorithmic}) to our candidate value
function given below. We adopt the second approach in the remaining of
this section.

We consider now the continuous execution problem on the time interval
\([0,T]\), where \(x_t\) is the trading speed in shares per unit of
time, \(v_t\) is the local market volume. In line with the discrete
problem, we associate to them the quantities \(u_t = x_t/v_t\) and
\(q_t = 1 - Q_0^{-1}\int_0^t x_s ds\). The price slippage
\(p_t = (S_t- \mathfrak S_t)/S_0\) used in the optimization now follows
the mean-reverting model
\[dp_t = -\kappa_t p_t dt + \mu_t x_t dt + \alpha_tdt+ \sigma_t dW_t\]
where \(W\) is a standard Brownian motion. The joint process
\(X_t = (p_t,q_t)\) thus follows the dynamics
\begin{align} \label{dynaXContinuous}
dX_t = -\kappa_t \mathcal J X_t dt + a_t u_t w_t dt+  \alpha_t e_1 dt +\sigma_t e_1dW_t
\end{align} where \(a_t = v_t/Q_0\),
\(w_t= \left( Q_0 \mu_t, - 1 \right)^T\) and with \(X_0 = (0, 1)^T\).
The marginal cost functions \(j_t(u_t,X_t)\) remain unchanged and the
global cost at time \(t\) starting from state \(X = (p,q)\) now takes
the form
\[ H(t,X) = \inf_{u \in \mathcal A_{t,T}} \mathbb{E} \left[\left.\int_{t}^{T} j_\tau(X_s,u_s) d s + J_T(X_T) \right.| X_t=X \right],\]
where \(\mathcal A_{t,T}\) is the set of predictable strategies. Given
(\ref{dynaXContinuous}), we readily derive the associated HJB equation
(see e.g (5.19) in \cite{cartea2015algorithmic})
\begin{align}\label{eqHJB}
\partial_t H(t,X) + \inf_{u \in \mathcal A_{t,T}} \left\{ \left[-\kappa_t \mathcal J X + a_t u_t w_t + \alpha_t e_1 \right]^T\partial_X H(t,X) + \frac{\sigma_t^2}{2} \partial_{p}^2 H(t,X) +  j_t(u_t,X) \right\} = 0,
\end{align} and the terminal condition
\[ H(T,X) = J_T(X) = \tilde{\beta}_T q^2.\] We are now ready to state
the main result of this section.

\begin{theorem} \label{thmLQRContinuous}
Let $X = (p,q)^T$, $\mathcal I = \text{diag}(1,1)$, $\mathcal{J} = \text{diag}(1,0)$, and $e_1 = (1,0)^T$. Assume that the backward Riccati differential equation (\ref{eqdiffP}) below admits a unique solution. Then, the value function $H$ is quadratic in $X$, of the form $H(t,X) = X^T P_t X + b_t^T X + e_t$. Moreover, the optimal control $u_t^*$ is of the form 
$$ u_t^* = k_t^T X_t + f_t $$
with 
$$k_t = -\frac{1}{\eta_t}\left(P_{t}^Tw_t + \frac{1}{2}e_1\right),$$
and
$$ f_t = -\frac{b_t^Tw_t + s_t}{2\eta_t}.$$

Moreover, we have the backward differential equations 
\begin{align} \label{eqdiffP}
- \partial_tP_t = - \kappa_t (\mathcal J P_{t} + P_{t} \mathcal J) + \Delta_t  - a_t\eta_tk_t k_t^T, 
\end{align}
\begin{align} 
-\partial_t b_t = - 2a_t\eta_tf_tk_t - \kappa_t \mathcal Jb_{t} + 2\alpha_tP_{t}e_1,
\end{align}
\begin{align} 
-\partial_t e_t =   \alpha_t b_t^T e_1 + \sigma_t^2 e_1^T P_{t}e_1 - a_t \eta_t f_t^2  
\end{align}
where $\Delta_t = \left( \begin{matrix}  0 &0  \\ 0 & \beta_t \end{matrix} \right)$, and with the terminal conditions $ P_T = \left( \begin{matrix} 0 &0  \\ 0 & \tilde{\beta}_T \end{matrix} \right)$, $b_T = (0,0)^T$, $e_T=0$. 
\end{theorem}

It is interesting to note that, as is often the case, the continuous
formulation somewhat simplifies the form of several terms involved in
the expression of the optimal strategy (only dominating terms in
\(\delta t\) do not vanish in the continuous trading limit). This
phenomenon is similar to what happens with the effect of linear
permanent impact on the optimization which disappears in the continuous
trading limit in the Almgren-Chriss framework as explained in, e.g
\cite{gueant2016financial}, p.61. However, due to the presence of
\(\kappa_t\) in our case, in general the optimal strategy obtained in
Theorem \ref{thmLQRContinuous} does depend on the permanent impact in
the continuous limit.

Another important feature of the above strategy is that it yields an
optimal solution whose expression depends on the matrix term \(P_t\)
which is the solution to the backward multi-dimensional Riccati equation
(\ref{eqdiffP}). It can be explicitly solved when \(\kappa_t = 0\),
since in that case the problem reduces to the canonical framework of
\cite{almgren2001optimal}, although in its continuous form, as exposed
in \cite{gueant2016financial}, pp.52-55. In general, however, if
\(\kappa_t \neq 0\), deriving a closed expression for \(P_t\) and
therefore for \(u_t^*\) seems challenging. From a practical point of
view, \(P_t\) can always be calculated using a backward Euler scheme or
any other first order numerical technique. There is one case, however,
which allows us to derive an explicit formula even in the presence of a
mean-reversion parameter, which we document in the next theorem.

\begin{theorem} \label{thmFormula}
Assume that there is no permanent impact $\mu_t=0$, and for simplicity that $s_t=0$, $\alpha_t=0$, and that moreover $\eta_t = \eta$, $a_t=a$, $\kappa_t=\kappa$, $\beta_t = \beta$ are all constant in time. Then the optimal strategy can be written as
$$ u_t^* = k_{t,1} p_t + k_{t,2} q_t$$
where
$$ k_{t,1} = - \frac{1+\zeta + \kappa \chi(t)}{2\eta\left(e^{-\sqrt{\frac{a\beta}{\eta}}(T-t)} +  \zeta e^{\sqrt{\frac{a\beta}{\eta}}(T-t)} \right)},$$
$$ k_{t,2} = \sqrt{\frac{\beta}{a\eta}}\frac{\zeta e^{2\sqrt{\frac{a\beta}{\eta}}(T-t)} -1}{\zeta e^{2\sqrt{\frac{a\beta}{\eta}}(T-t)} +1},$$
$$ \zeta = \frac{ \sqrt{\frac{\eta \beta}{a}} + \tilde{\beta}_T }{ \sqrt{\frac{\eta \beta}{a}} - \tilde{\beta}_T},$$
$$ \chi(t) = \frac{1}{\kappa - \sqrt{\frac{a\beta}{\eta}}} \left( e^{\left(\kappa - \sqrt{\frac{a\beta}{\eta}} \right)(T-t)} -1\right) + \frac{\zeta}{\kappa + \sqrt{\frac{a\beta}{\eta}}} \left( e^{\left(\kappa + \sqrt{\frac{a\beta}{\eta}} \right)(T-t)} -1\right).$$
\end{theorem}

We now give several comments about the above results.

\begin{itemize}
\tightlist
\item
  The formula for \(k_{2,t}\) is independent from \(\kappa\), and
  corresponds to what is obtained in the literature in the absence of
  mean-reversion, see for instance Section 6.5 in
  \cite{cartea2015algorithmic}.
\item
  The coefficient in front of the price slippage, \(k_{1,t}\) is
  affected by \(\kappa\) through the term \(\kappa\chi(t)\). A
  straightforward analysis shows that under the hard completion
  constraint \(\tilde{\beta}_T \to +\infty\), that is \(\zeta = -1\),
  \(\chi\) is a non-positive increasing function with terminal condition
  \(\chi(T) = 0\). From this one can conclude two things. First, because
  \(\chi\) is non-positive, we see that \(\kappa\) has a negative impact
  on the relationship between the optimal participation rate and the
  price slippage, which confirms the fact that the algorithm is
  aggressive in the money from a theoretical perspective. Second, since
  \(\chi\) decreases in absolute value over time to finally vanish at
  time \(T\), the effect of mean-reversion on the curve is the strongest
  at the beginning of the order and then continuously decreases as time
  passes. This is in line with the intuition that as we approach \(T\),
  there is no time left to benefit from the mean-reversion effect while
  the completion constraint becomes predominant.
\end{itemize}

We finally specify two important particular cases in the next corollary.

\newtheorem{corollary}{Corollary} 
\begin{corollary}\label{CorFormula}
In the absence of inventory costs ($\beta = 0$), the above expressions are degenerated and yield
$$ k_{1,t} = - \frac{ e^{\kappa(T-t)}+ \frac{a\tilde{\beta}_T}{\eta} \left((T-t)e^{\kappa(T-t) }- \frac{1}{\kappa}\left(e^{\kappa(T-t)}-1\right)\right)}{2\eta + 2\tilde{\beta}_Ta(T-t)}$$
$$ k_{t,2} = \frac{\tilde{\beta}_T}{\eta+a\tilde{\beta}_T(T-t)}.$$

Under the hard constraint $\tilde{\beta}_T \to +\infty$ and under no inventory costs $\beta =0$, we obtain on $[0,T)$
$$ k_{t,1} = -\frac{e^{\kappa(T-t) }- \frac{1}{\kappa(T-t)}\left(e^{\kappa(T-t)}-1\right)}{2\eta},$$
$$ k_{t,2} = \frac{1}{a(T-t)},$$
and $k_{T,1}= -\frac{1}{2\eta}$,$k_{T,2} = +\infty$. If moreover the mean-reversion signal is weak, then $\kappa(T-t) \ll 1 $ for $t \neq T$ and we have 
$$ k_{t,1} \sim -\frac{\kappa}{2\eta}(T-t).$$
\end{corollary}

Figure 4 shows the shape of \(k_{t,1}\) when there are no inventory
costs and under the hard constraint of forced completion. We see that,
as expected, \(k_{1,t}\) decreases in absolute value as \(t\) increases.

\begin{figure}
\centering
\includegraphics{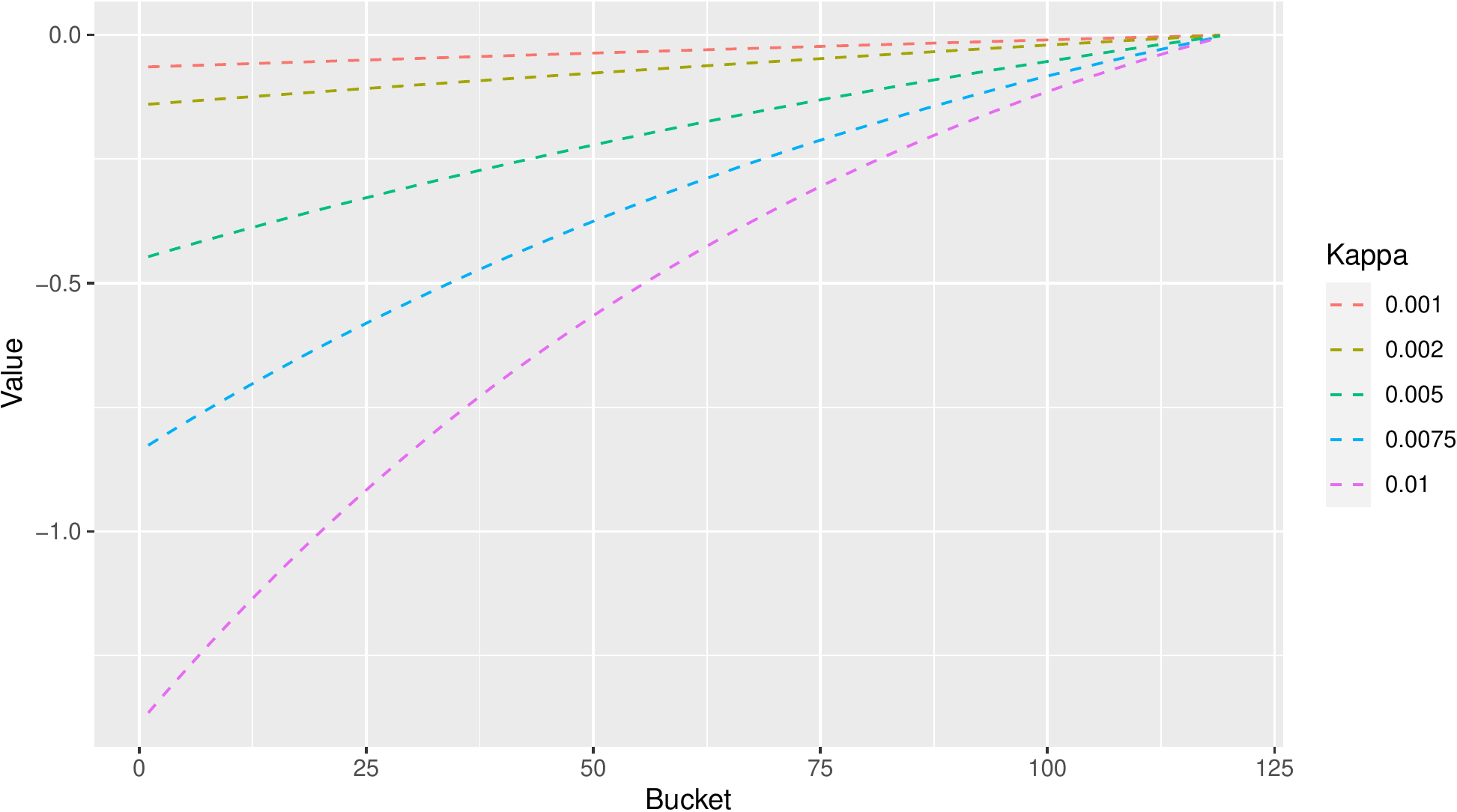}
\caption{price coefficient for different values of kappa}
\end{figure}

\section{LQR with an additional dark pool}

We finally examine the case where shares can be traded in a second
exchange called \textit{dark pool}. The liquidation problem becomes more
complex because both platforms offer concurrent prices and market impact
rules. This calls for an additional level of optimization when deciding
what fractions of shares should be sent to the lit (i.e ordinary) market
and the dark pool. We first describe the underlying mechanisms driving
the trading process in an idealized dark pool in the next paragraph,
following mainly the model of \citep{cartea2015algorithmic}, Section
7.4, with slight modifications. We then bring forward a two-step
strategy which simultaneously posts orders in both markets.

\subsection{Model}

Dark pools are special trading venues where bid and ask quotes are not
revealed to their clients, and where trades can only occur at the
mid-price in the corresponding lit market. An important consequence of
this rule is that, by posting an order during the bucket
\([t, t+\delta t)\), one can never be sure whether the order will
actually be executed. On the other hand, the price execution will be, in
theory, the mid-price \(S_t\), which is more favorable than what is
offered on the lit market at the same time by a quantity
\(S_0(\eta_t u_t + s_t)\). In practice, however, having orders placed in
the dark pool does send a signal to the other market participants since
they can be uncovered by probing the available liquidity (e.g by sending
small orders to the dark pool and checking whether they trigger a
trade). This, in turn, has an impact on the mid-price due to adverse
selection. In any case, dark pools can be seen as venues that offer
better price slippages than lit markets, at the cost of higher variance
due to the uncertain nature of execution.

Let \(y_t\) be the amount of shares the algorithm has posted at some
point and which are still pending for execution in the dark pool during
the bucket \([t,t+\delta t)\). In contrast with the lit market, \(y_t\)
is expressed in absolute shares because the market volume in the dark
pool is never observed. Note also that, for instance, if
\(y_t = y_{t+\delta t}\) and no quantity was executed during
\([t, t + \delta t)\), then no order should be sent to the dark pool at
time \(t + \delta t\) since the quantity posted at \(t\) is still
pending for execution. We assume that over the bucket
\([t, t+\delta t)\)

\begin{itemize}
\tightlist
\item
  the pending quantity \(y_t\) is executed \emph{in full} with a small
  probability \(\lambda_{t}^D \delta t\). Realistic values for
  \(\lambda_{t}^D \delta t\) with 1 to 5-minute buckets are around
  1-10\%. We write \(\epsilon_t^D\) the associated variable, where
  \(\epsilon_t^D = 1\) means that \(y_t\) is executed whereas
  \(\epsilon_t^D=0\) indicates that nothing happened. In the continuous
  limit, this is equivalent to assuming that orders placed in the dark
  pool are matched at random times following an inhomogeneous Poisson
  process of intensity \(\lambda_t^D\).
\item
  The adverse selection effect described above amounts to setting the
  paid-price to
  \[ \tilde{p}_t = p_t + s_t + \eta_t u_t + \theta_t^D y_t\] for orders
  in the lit market and
  \[ \tilde{p}_t^D = p_t  + \eta_t u_t + \theta_t^D y_t\] in the dark
  pool, where \(\theta_t^D\) is an impact parameter. In realistic cases,
  \(\theta_t^D\) should be taken so that \(\theta_t^D v_t\) is
  significantly smaller than \(\eta_t\), since the absence of direct
  quotes information on dark pools tends to mitigate the impact orders
  have when posted.
\end{itemize}

The first point stresses the radical difference between lit and dark
markets. While execution in the lit is done continuously by trading
small amounts during each bucket, execution in the dark yields long
periods with no execution at all separated by sudden spikes of large
executions. In continuous time, they make the execution profile jump by
a significant fraction a few times during an execution (see Figure 5).

\subsection{Two-step optimization}

We look for bivariate strategies \((u_t^*, y_t^*)\) that adequately
balance the fraction of shares sent to both venues over time. Since the
impact term \(\theta_t^D y_t\) and the executed quantity in dark
\(\epsilon_t^D y_t\) are linear, it is possible to solve a multivariate
LQR to obtain an explicit optimal strategy. Unfortunately, our
experience with this approach has shown that the related solution is
numerically unstable, yielding spikes of large quantities with opposite
signs posted simultaneously in the lit and the dark markets, even when
\(\theta_t^D, \eta_t, \mu_t\) and \(\lambda_t^D\) are calibrated so that
no arbitrage is permitted between both venues. What is more,
constraining the strategies by MPC is numerically highly inefficient due
to the presence of the Poisson process in the optimization problem. In
particular, the MPC-based strategy is already strongly biased in the
absence of constraints.

Accordingly, we set aside the global problem in this paragraph, and give
an approximating two-step strategy built on the previous sections. The
process goes as follows. For each bucket \([t, t+ \delta t)\), we run
the MPC scheduler on the lit market alone, and get a participation rate
\(u_t^*\). We then choose a fraction \(\mathfrak{a} \in [0,1]\) of the
remaining quantity \(q_t - a_t u_t^*\) to execute by \(T\) that
corresponds to the maximal quantity that can be allocated to the dark
pool during the current bucket. Given this quantity and the current
price slippage, we calculate an optimal allocation \(y_t^*\) in the LQR
sense, following the rules of the previous paragraph. We now give the
mathematical details behind the second step.

Assume that the procedure starts at time \(0 \leq t \leq T\). Let us
define \(q_t^D = \mathfrak{a}(q_t - a_t u_t^*)\), the scaled quantity
passed to the dark LQR optimizer. The price slippage is, as for the lit
market, simply \(p_t\), and in a similar fashion as before, we let
\(X_t^D\) be the state process \(\left(p_t,q_t^D\right)\). From the dark
LQR perspective, everything is as if \(q_t^D\) was the remaining
quantity to execute by the dark only by \(T\). Next, we define the price
paid in the dark venue as
\[ \tilde{p}_s^D = p_s + \eta_s \langle u\rangle_t + \theta_s^D y_s\]

for any \(t \leq s \leq T\). The average participation rate over the
period \(t, t+ \delta t,...,s,...T\) in the lit market
\(\langle u \rangle_t\) can be taken as
\((1-\mathfrak a)Q_t/\left(\sum_{\tau=t}^T v_\tau \delta t\right)\).
Note that it is not possible to substitute to this value the actual
participation \(u_s^*\) since it is not known at this point as soon as
\(s > t\). The term \(\eta_s \langle u \rangle_t\) therefore accounts
for the impact coming from the lit market and cannot be controlled. The
associated stage cost is \begin{align}  \label{eqStageCostDark}
j_s^D(X_s^D,y_s)\delta t = \left\{ (\delta t)^{-1}\frac{a_s}{v_s}\theta_s^D \epsilon_s^D  y_s^2 + (\delta t)^{-1}\frac{a_s}{v_s}p_s \epsilon_s^D y_s   + \left( \theta_s^D + (\delta t)^{-1}\frac{\eta_s}{v_s} \epsilon_s^D \right)  a_s\langle u \rangle_t y_s +\beta_s^D \left(q_s^D\right)^2 \right\} \delta t
\end{align} where the term \(\theta_s^D a_s\langle u \rangle_t y_s\)
represents the supplementary cost paid on the lit market due to the fact
that posting orders in the dark pool pushes the mid-price. The term
\((\delta t)^{-1}\frac{\eta_s}{v_s} \epsilon_s^D a_s\langle u \rangle_t y_s\),
on the other hand, accounts for the impact from the lit market on the
dark quantity, paid when there is an execution in the dark. The terminal
cost is simply \begin{align}  \label{eqTerminalCostDark}
j_T^D \left(X_T^D\right) = \tilde{\beta}_T \left(q_T^D\right)^2. 
\end{align} In general, in contrast with the lit market, it is not
necessary to take a high value for \(\tilde{\beta}_T^D\) since the lit
MPC will complete the order no matter what happens if
\(\tilde{\beta}_T = + \infty\). The state space equation becomes, in
matrix form \begin{align} \label{eqState2Dark}
X_{s+\delta t}^D = \left( \begin{matrix}  1- \kappa_s \delta t &0  \\ 0 &1 \end{matrix} \right)  X_s^D - \frac{a_s}{v_s} \epsilon_s^D e_2 y_s +  \left( \begin{matrix}  \alpha_s \delta t+\sigma_{t, \delta t} \xi_s \\ 0 \end{matrix} \right).
\end{align}

We finally define the score function at point \(s\) as
\[ H^D(s,X^D) = \inf_{(y_k)_{k = s,...,T-\delta t}} \mathbb{E} \left[\left.\sum_{\tau=s}^{T-\delta t} j_\tau^D(X_\tau^D,y_\tau)\delta t + J_T^D(X_T^D) \right| X_t^D=X^D \right].\]

\begin{theorem} \label{thmLQRDark} (Unconstrained LQR with pure dark pool execution)
Let $X^D = (p,q^D)^T$, $\mathcal I = \text{diag}(1,1)$, $\mathcal{J} = \text{diag}(1,0)$, and $e_1 = (1,0)^T$. The value function $H^D$ starting from $t$  with $0 \leq t \leq T$, is quadratic in $X^D$, of the form 
$$H^D(s,X^D) = (X^D)^T P_s^D X^D+ (b_s^D)^T X^D + e_s^D$$
for any $t \leq s \leq T$. The optimal control $y_s^*$ is of the form 
\begin{align}  \label{eqUDark}
y_s^* = k_s^D X^D + f_s^D 
\end{align}
with 

$$ g_s^D =  \theta_s^D   + \frac{ a_s}{v_s} e_2^T P_{s + \delta t}^D e_2  $$

$$(k_s^D)^T = -  (g_s^D)^{-1}    \left[ -e_2^T (P_{s + \delta t}^D)^T (\mathcal I -\kappa_s \delta t \mathcal J) + \frac{1}{2} e_1^T \right]   $$
and
$$ f_s^D = -\frac{1}{2} (g_s^D)^{-1} \left[-e_2^Tb_{s + \delta t}^D- 2\alpha_s \delta t e_1^TP_{s + \delta t}^De_2 +\left(\frac{v_s\theta_s^D}{\lambda_s^D} + \eta_s \right) \langle u \rangle_t\right]. $$
Moreover, we have the backward iteration scheme 
$$ P_s^D =  P_{s + \delta t}^D - \kappa_s \delta t (\mathcal J P_{s + \delta t}^D + P_{s + \delta t}^D \mathcal J) + \kappa_s^2 (\delta t)^2 \mathcal J P_{s + \delta t}^D \mathcal J + \left[\Delta_s^D  - \frac{a_s \lambda_s^D}{v_s}g_s^D k_s^D  (k_s^D)^T\right]\delta t,$$
$$ b_s^D = b_{s + \delta t}^D -\left\{2\frac{a_s \lambda_s^D}{v_s} g_s^D f_s^D k_s^D  + \kappa_s \mathcal Jb_{s + \delta t}^D  -2\alpha_s (\mathcal I - \kappa_s\delta t \mathcal J)   P_{s + \delta t}^De_1 \right\}\delta t,$$
$$e_s^D = e_{s + \delta t}^D + [\sigma_{s,\delta t}^2+\alpha_s^2 (\delta t)^2]  e_1^T P_{s + \delta t}^De_1   +\left[\alpha_s (b_{s + \delta t}^D)^Te_1 - \frac{a_s \lambda_s^D}{v_s}   g_s^D  (f_s^D)^2\right] \delta t $$
where $\Delta_s^D = \left( \begin{matrix}  0 &0  \\ 0 & \beta_s^D \end{matrix} \right)$, and with the terminal conditions $ P_T^D = \left( \begin{matrix} 0 &0  \\ 0 & \tilde{\beta}_T^D \end{matrix} \right)$, $b_T^D = (0,0)^T$, $e_T^D=0$. 
\end{theorem}

The implementation of the two-step strategy is documented in Algorithm
\ref{algoTwoStep}.

\begin{algorithm}[H] \label{algoTwoStep}
 \KwOut{The optimal Lit/Dark trading curve $(u_0^*, y_0^*)$, $(u_{\delta t}^*,y_{\delta t}^*), \cdots, (u_{T-\delta t}^*,y_{T-\delta t}^*)$.}
  Initialize clock time $\tau \leftarrow 0$ \;
  \For{$t = 0$ to $T-\delta t$}{
  Wait for clock time $\tau \geq t$\;
  Fetch $Q_t$, $\mathfrak{S}_t$, $S_t$\;
  Calculate $u_t^*$ by lit MPC \;
  $\langle u \rangle_t \leftarrow (1-\mathfrak a)Q_t/\left(\sum_{\tau=t}^T v_\tau \delta t\right)$\;
  $q_t^D \leftarrow \mathfrak a (q_t - a_t u_t^*)$ \;
  Calculate $y_t^*$ by dark LQR \;
  $y_t^* \leftarrow \min(\max(y_t^*, 0), q_t^D Q_0)$\;
  }
 \caption{Online calculation for the two-step optimal policy $(u_t^*, y_t^*)_{0 \leq t \leq T - \delta t}$.}
\end{algorithm}

We now briefly give the continuous time version of Theorem
\ref{thmLQRDark}. In the continuous limit the state process evolution
equation becomes
\[ dX_t^D = -\kappa_t \mathcal J X_t^D dt   +\alpha_t e_1dt - a_t \frac{y_{t}}{v_{t}}e_2 dN_t + \sigma_t e_1 dW_t\]
where \(N\) is an inhomogeneous Poisson process of intensity
\(\lambda_t^D\), and \(y_t\) and \(v_t\) are predictable. Let
\[ H^D(s,X^D) = \inf_{y \in \mathcal A_{s,T}^D} \mathbb{E} \left[\left.\int_{s}^{T} j_\tau^D(X_v^D,y_v) dv + J_T^D(X_T^D) \right| X_t^D=X^D \right]\]
where \(A_{s,T}^D\) is the set of predictable strategies. The associated
HJB equation is (Theorem 5.4, \citep{cartea2015algorithmic})
\begin{eqnarray*}\label{eqHJBDark}
\partial_t H^D(s,X^D) &+& \inf_{y \in \mathcal A_{s,T}^D} \left\{ \left[-\kappa_s \mathcal J X  \right]^T\partial_X H^D(s,X^D) + \frac{\sigma_s^2}{2} \partial_{p}^2 H^D(s,X^D) \right. \\
&+& \left. \lambda_s^D \left\{ H^D\left(s,X^D - a_t\frac{y_s}{v_s}e_2\right) - H^D(s,X^D)\right\} + \tilde{j}_s^D(X^D,y_s) \right\} = 0
\end{eqnarray*} with
\[ \tilde{j}_s^D(X^D,y_s) =  \frac{a_s}{v_s}\theta_s^D \lambda_s^D  y_s^2 + \frac{a_s}{v_s}p \lambda_s^D y_s   + \left( \theta_s^D + \frac{\eta_s}{v_s} \lambda_s^D \right)  a_s\langle u \rangle_t y_s + \beta_s^D \left(q^D\right)^2,\]
and with the terminal condition
\[ H^D(T,X^D) = \tilde{\beta}_T^D \left(q^D\right)^2.\] We obtain the
following theorem.

\begin{theorem} \label{thmLQRcontDark} (Unconstrained LQR with dark pool execution, continuous version)
Let $X^D = (p,q^D)^T$, $\mathcal I = \text{diag}(1,1)$, $\mathcal{J} = \text{diag}(1,0)$, and $e_1 = (1,0)^T$. The value function $H^D$ starting from $t \in [0,T)$, is quadratic in $X$, of the form 
$$H(s,X^D) = (X^D)^T P_s^D X^D + (b_s^D)^T X^D + e_s^D$$
for any $s \in [t, T)$. The optimal control $y_s^*$ is of the form 
\begin{align}  \label{eqUContDark}
y_s^* = k_s^D X + f_s^D 
\end{align}
with 

$$ g_s^D =      \theta_s^D   + \frac{ a_s}{v_s} e_2^T P_{t}^De_2  $$

$$(k_s^D)^T = -  (g_s^D)^{-1}    \left[ -e_2^T (P_{t}^D)^T   + \frac{1}{2} e_1^T \right]   $$
and
$$ f_s^D = -\frac{1}{2} (g_s^D)^{-1} \left[-e_2^Tb_{s }^D+\left(\frac{v_s\theta_s^D}{\lambda_s^D} + \eta_s \right) \langle u \rangle_t\right] $$
Moreover, we have the backward differential equations
$$ -\partial P_s^D =   - \kappa_s  (\mathcal J P_{s }^D + P_{s }^D \mathcal J) +  \Delta_s^D  - \frac{a_s \lambda_s^D}{v_s}g_s^D k_s^D  (k_s^D)^T ,$$
$$ -\partial b_s^D = - 2\frac{a_s \lambda_s^D}{v_s} g_s^D f_s^D k_s^D  - \kappa_s \mathcal Jb_{s }^D  +2\alpha_s   P_{s}^D e_1  ,$$
$$-\partial e_s^D =   \sigma_{s}^2 e_1^T P_{s }^De_1   + \alpha_s (b_{s}^D)^Te_1  - \frac{a_s \lambda_s^D}{v_s}   g_s^D  (f_s^D)^2  $$
where $\Delta_s^D = \left( \begin{matrix}  0 &0  \\ 0 & \beta_s^D \end{matrix} \right)$, and with the terminal conditions $ P_T^D = \left( \begin{matrix} 0 &0  \\ 0 & \tilde{\beta}_T^D \end{matrix} \right)$, $b_T^D = (0,0)^T$, $e_T^D=0$. 
\end{theorem}

Figure 5 shows an example of a lit MPC/dark LQR execution, with the
signal \(\kappa_t > 0\). We also have \(\mathfrak a = 50\%\),
\(\lambda_t^D \delta t = 1/30\), market volumes are flat, there is no
urgency during the algorithm, the hard completion constraint \(q_T =0\)
for the lit market, and a small completion constraint
\(\tilde{\beta}_T^D \approx 0\) for the dark pool.

\begin{figure}
\centering
\includegraphics{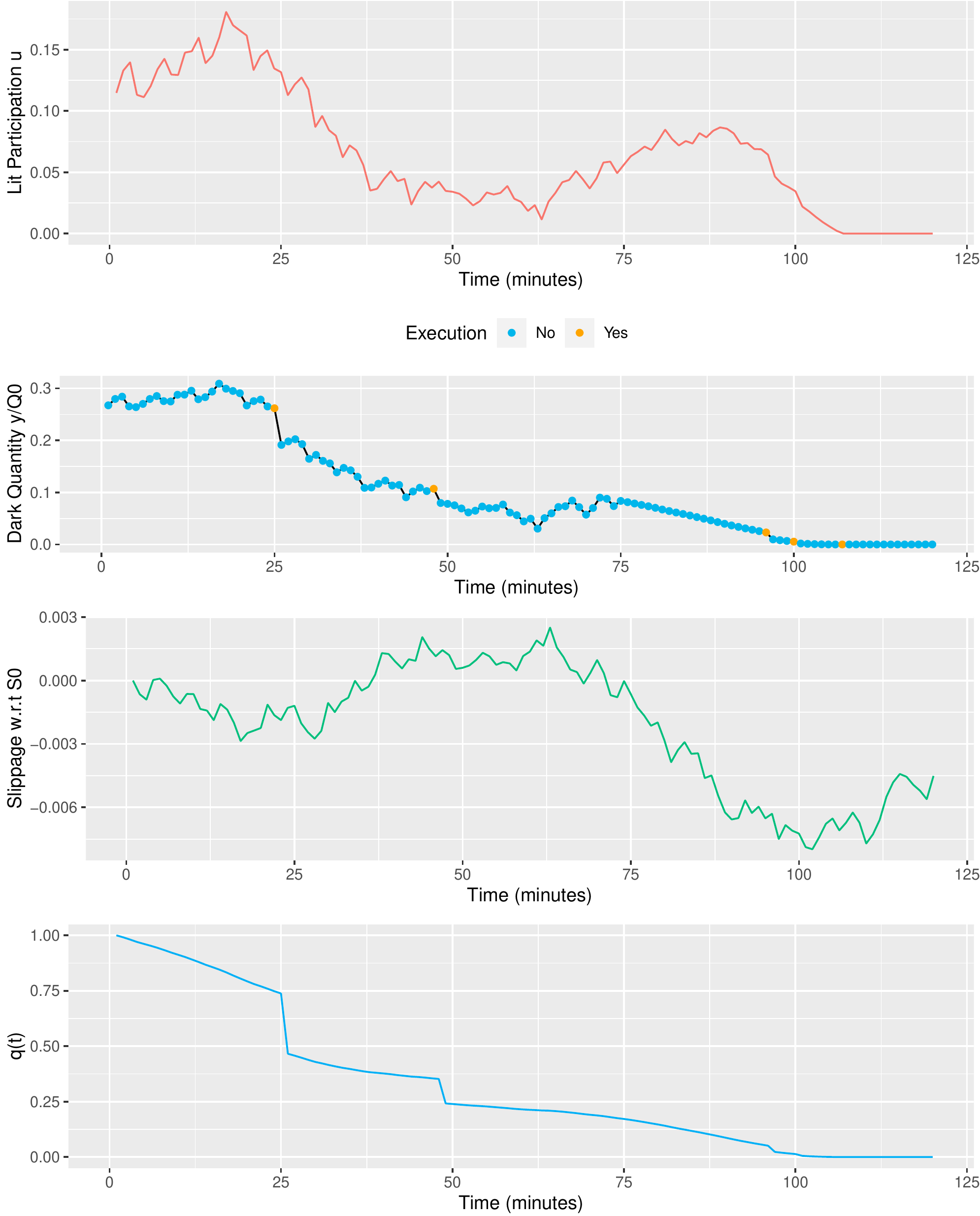}
\caption{Optimal execution with a lit market and a dark pool}
\end{figure}

\section{Conclusion} \label{SectionConclusion}

We have shown that by adding a mean-reversion signal in the price
slippage dynamics, we obtain an optimal execution problem following an
LQR specification. The exact shape of the optimal policy function can be
iteratively derived before the order starts, and is aggressive in the
money. If round-trip trades must be avoided, the MPC approach can impose
positivity constraints on the participation curve while remaining close
to the LQR optimal strategy. In the continuous trading limit with no
permanent impact, we have seen that one can derive a closed formula for
the optimal policy which can be further simplified under no urgency,
full completion condition, and when the mean-reversion signal is weak.
Finally, we have seen that it is possible (although heuristic) to
construct a two-step strategy in the presence of an additional dark
pool, where a lit MPC policy and a dark LQR strategy work together.

The mean-reversion signal can be seen from different perspectives. It
can be an actual, local mean-reversion signal that one has estimated on
a particular asset. However, it is often difficult to find reliable
patterns in price series. Still, such signals may exist, in particular
if the execution task concerns not one but a pair of stocks or more. As
exposed in Section \ref{SectionKappa}, another approach consists in
using the mean-reverting signal as a tuning parameter in the optimizer
to induce a certain level of aggressiveness in the money. This can be
motivated by the broker's inside information or by clients who have
their own agenda.

We leave for future research the case of multi-asset optimal execution
and rolling MPC as exposed at the end of Section \ref{SectionMPC}.
Although they are of great importance for more advanced practical
applications, we believe that they are complex problems and they both
deserve more than a mere section in the present paper.

\section{Appendix}

\subsection{Proofs of Theorem \ref{thmLQR} and Theorem \ref{thmLQR2}}

We prove only Theorem \ref{thmLQR2}, since Theorem \ref{thmLQR}
corresponds to the particular case \(\alpha_t = 0\). Let
\(\mathcal{I} = \text{diag}(1,1)\), and
\(\mathcal{J} = \text{diag}(1,0)\). Recall the Bellman equation
\[ H(t, X) = \inf_{u \in \mathbb R} \left\{j_t(X, u) \delta t + \underbrace{\mathbb E \left[ H\left(t+ \delta t,  X - \kappa_t \delta t \mathcal{J} X + a_t w_t \delta t u+  \left( \begin{matrix} \alpha_t \delta t + \sigma_{t, \delta t} z_t \\ 0 \end{matrix} \right) \right)\right]}_{h_t(X,u)} \right \}.\]
We prove our claim on \(H(t,X)\) and \(u_t^*\) by descending induction
on \(t \in \{ 0, \delta t, ..., T\}\). At time \(T\), there is no
control, and \(H(T,X) = \tilde{\beta}_T q_T^2\) which is the claimed
quadratic form in \(T\), with \(P_T\), \(b_T^T\) and \(e_T\) as stated
in Theorem \ref{thmLQR}. Assume now that the claim is true on
\(t+\delta t, ..., T\). Then we have
\[ j_t(X,u) = a_t\eta_t u^2 + a_t(p+s_t)u + \beta_t q^2 \] and direct
calculations using
\(H(t+\delta t, X) = X^T P_{t+\delta t} X + b_{t+\delta t}^T X + e_{t+\delta t}\)
yield \begin{eqnarray*} 
h_t(X,u) &=& a_t^2 \delta t^2 w_t^T P_{t+\delta t}w_t  u^2 + \left \{2 a_t    \left (X^T  \left(\mathcal{I} - \kappa_t \delta t \mathcal J\right) + \alpha_t \delta t e_1^T\right) P_{t+\delta t} w_t + a_tb_{t+\delta t}^Tw_t \right\} \delta tu \\ &+&  X^T\left(\mathcal{I} - \kappa_t \delta t \mathcal J\right) P_{t+\delta t} \left(\mathcal{I} - \kappa_t \delta t \mathcal J\right) X + \left[b_{t+\delta t}^T + 2 \alpha_t \delta t e_1^T P_{t+\delta t} \right]\left(\mathcal{I} - \kappa_t \delta t \mathcal J\right)X \\&+& e_{t+\delta t} + \alpha_t \delta t b_{t+\delta t}^Te_1 +  \left[\sigma_{t,\delta t}^2 + \alpha_t^2 (\delta t)^2\right] e_1^T P_{t + \delta t}e_1.
\end{eqnarray*} Hence \(j_t(X,u)\delta t+h_t(X,u)\) is quadratic in
\(u\) and we readily check that the optimal \(u\) which makes it minimal
is given by \(k_t^T X + f_t\). We then directly get that plugging this
is in \(j_t(X,u)\delta t+h_t(X,u)\) yields the quadratic expression in
\(X\) \[H(t,X) = X^T P_t X + b_t^T X +e_t\] with
\[ P_t =  P_{t+\delta t} - \kappa \delta t (\mathcal J P_{t+\delta t} + P_{t+\delta t} \mathcal J) + \kappa^2 (\delta t)^2 \mathcal J P_{t+\delta t} \mathcal J + \left[\Delta_t  - a_tg_tk_t k_t^T\right]\delta t,\]
\[ b_t = b_{t+\delta t} -\left\{2a_tg_tf_tk_t + \kappa_t \mathcal Jb_{t+\delta t}  - 2\alpha_t  (\mathcal I - \kappa_t \delta t\mathcal J)P_{t+\delta t} e_1 \right\}\delta t,\]
\[e_t = e_{t+\delta t} +  [\sigma_{t,\delta t}^2+ \alpha_t^2(\delta t)^2] e_1^T P_{t+\delta t}e_1 + \left[  \alpha_t  b_{t+\delta t}^Te_1   - a_t g_t f_t^2 \right]  \delta t \]
as claimed in the theorem.

\subsection{Proofs of Theorem \ref{thmMpc} and Proposition \ref{propMPC}}

\begin{proof}[Proof of theorem \ref{thmMpc}]
We derive the expression of the quadratic form $C_t(u_t, u_{t+\delta t}, ..., u_{T-\delta t})$. To do so, it is convenient to first extend the state process $X_t = (p_t,q_t)^T$ to $\tilde{X}_t = (p_t,q_t,u_t)^T$. At time $0 \leq t \leq T$, the stage cost  $j_t(X_t,u_t)\delta t$ can be rewritten as a quadratic form in $\tilde{X}_t$. We have 
$$ j_t(\tilde{X}_t) = \tilde{X}_t^T M_t \tilde{X}_t + N_t \tilde{X}_t$$
where we recall that for $t< T$ we have
$$M_t = \left(\begin{matrix} 0 &0&a_t/2\\0&  \beta_t& 0 \\a_t/2&0&a_t\eta_t\end{matrix}\right), \text{ and } N_t =   \left(\begin{matrix}0\\0\\ a_ts_t\end{matrix}\right)$$
and at time $T$ 
$$M_T = \left(\begin{matrix} 0 &0&0\\0&  \tilde{\beta}_T& 0 \\0&0&0\end{matrix}\right), \text{ and } N_t =   \left(\begin{matrix}0\\0\\0\end{matrix}\right).$$
Now, the global cost starting from time $t$ is 
\begin{eqnarray*}
C_t(u_t, u_{t+\delta t}, ..., u_{T-\delta t}) = \sum_{\tau = t}^{T-\delta t} j_\tau(\tilde{X}_\tau)\delta t +J_T(\tilde{X}_T).
\end{eqnarray*}
Note the linear representation, under CE,
\begin{eqnarray*}  \tilde{X}_\tau &=& \underbrace{\left(\begin{matrix} \pi_{(t+\delta t):(\tau-\delta t)}^\kappa v_t\mu_{t}\delta t & \pi_{(t+2\delta t):(\tau-\delta t)}^\kappa v_{t+\delta t}\mu_{t+\delta t}\delta t &...&  v_{\tau-\delta t}\mu_{\tau-\delta t}\delta t & 0 &0& ...\\  -a_{t}\delta t & -a_{t+\delta t}\delta t& ... &   -a_{\tau-\delta t}\delta t&0 & 0&... \\0&0&...&0&1&0&...\end{matrix}\right)}_{L_{\tau}} u_{t:(T-\delta t)} \\&+& \underbrace{\left(\begin{matrix}   \pi_{t:(\tau-\delta t)}^\kappa p_{t} + \sum_{\theta = t}^{\tau-\delta t} \pi_{\theta:(\tau-\delta t)}^\kappa \alpha_\theta \delta t\\q_{t} \\0\end{matrix} \right)}_{Y_{\tau}},
\end{eqnarray*}
so that

\begin{eqnarray*} 
C_t(u_t, u_{t+\delta t}, ..., u_{T-\delta t}) &=& u_{t:(T-\delta t)}^T \left(\left\{\sum_{\tau = t}^{T-\delta t} L_\tau^T M_\tau L_\tau \right\}\delta t + L_T^T M_T L_T\right) u_{t:(T-\delta t)} \\ &+& \left(\left\{\sum_{\tau = t}^{T-\delta t} 2 Y_{\tau} M_\tau L_\tau + N_\tau L_\tau \right\} \delta t + 2 Y_{T} M_T L_T + N_T L_T \right) u_{t:(T-\delta t)}  \\&+& \textbf{c}_t,
\end{eqnarray*}
where $\textbf{c}_t$ is a constant term,  as claimed in the theorem.
\end{proof}

\begin{proof}[Proof of Proposition \ref{propMPC}]
Remark that the coefficients $k_t^T$ and $f_t$ do not depend on $\sigma_{t, \delta t}^2$ in Theorem \ref{thmLQR2}. Hence they would be unchanged if, all things being equal, $\sigma_{t, \delta t}^2$ were set to $0$. Therefore, the LQR solver yields the optimal policy under the certainty equivalent problem. On the other hand, by definition, the MPC solver in the absence of constraints precisely gives the optimal participation rate under CE.   

\end{proof}
\subsection{Proofs of Theorem \ref{thmLQRContinuous},  Theorem \ref{thmFormula} and Corollary \ref{CorFormula}}

\begin{proof}[Proof of Theorem \ref{thmLQRContinuous}]
We look again for a quadratic value function of the form $H(t,X) = X^T P_t X + b_t^T X +e_t$ and will use the verification theorem to back up our claim. Similarly to the discrete case, plugging $H(t,X)$ in (\ref{eqHJB}) and minimizing the quadratic expression in $u_t$ easily yields the expression of the optimal strategy in feedback form 
$$ u_t^* = k_t^T X_t + f_t.$$
Next, evaluating  (\ref{eqHJB}) at point $u_t^*$ and grouping together quadratic terms in $X$, then linear terms in $X$, and finally the constants respectively give the claimed backward differential equations for $P_t$, $b_t$, and $e_t$. By assumption $P_t$ is well-defined, and so are $b_t$ and $e_t$ which follow first order linear differential equations. Finally, by construction $H$ satisfies Theorem 5.2 in \cite{cartea2015algorithmic}.  
\end{proof}

\begin{proof}[Proof of Theorem \ref{thmFormula} and Corollary \ref{CorFormula}]
Let us write $P_t = (p_{ij,t})_{i,j \in \{1,2\}}$ with $p_{12,t} = p_{21,t}$. Under the assumption of the theorem, note that we have 
$$ k_t = \frac{1}{\eta}\left( p_{12,t} - \frac{1}{2}, p_{22,t} \right)^T$$
where $p_{12,.}$ and $p_{22,.}$ satisfy the following Riccati equations 
$$ \partial_t p_{12,t} = \kappa p_{12,t}+\frac{a}{\eta} p_{22,t} \left(p_{12,t} - \frac{1}{2}\right) $$
and 
$$ \partial_t p_{22,t} = -\beta + \frac{a}{\eta} p_{22,t}^2$$
with terminal conditions $p_{22,T} = \tilde{\beta}_T$ and $p_{12,T} = 0$. Solving for $p_{22,.}$ first note that we have 
$$ -\frac{\eta}{a}\frac{\partial_{t} p_{22,t}}{ \frac{\eta\beta}{a} - p_{22,t}^2} =1$$
so that integrating on both sides on $[t,T]$ and using that $\frac{1}{2\sqrt{\frac{\eta \beta}{a}}}\ln\left( \frac{\sqrt{\frac{\eta \beta}{a}} + p_{22,.}}{\sqrt{\frac{\eta \beta}{a}} - p_{22,.}}\right)$ is a primitive for $\frac{\partial_{t} p_{22,.}}{ \frac{\eta\beta}{a} - p_{22,.}^2}$, we obtain
$$ - \frac{\eta}{a} \left[ \ln\left( \frac{\sqrt{\frac{\eta \beta}{a}} + p_{22,T}}{\sqrt{\frac{\eta \beta}{a}} - p_{22,T}}\right) - \ln\left( \frac{\sqrt{\frac{\eta \beta}{a}} + p_{22,t}}{\sqrt{\frac{\eta \beta}{a}} - p_{22,t}}\right)\right] = 2(T-t)\sqrt{\frac{\eta \beta}{a}}.$$
Replacing $p_{22,T}$ by $\tilde{\beta}_T$ and isolating $p_{22,t}$ in the above equation readily yields the claimed representation. Next, we introduce $y(t) = p_{12,t} - \frac{1}{2}$ and note that  we have 
$$ \partial_t y(t) = \left(\kappa +\frac{a}{\eta} p_{22,t} \right) y(t) + \frac{\kappa}{2}, \text{ and }y(T) = -\frac{1}{2}.$$
Let $y_0$ be the solution to the homogenous equation (without the last term $\kappa/2$) satisfying $y_0(T) = 1$. Then we have 
$$\frac{ \partial_t y_0(t)}{y_0(t)} = \kappa +\frac{a}{\eta} p_{22,t}$$
so that integrating on both sides on $[t,T]$ yields again  
$$ \ln \frac{1}{y_0(t)} = \kappa(T-t) - \sqrt{\frac{a\beta}{\eta}} \int_t^T \frac{\zeta e^{2\sqrt{\frac{a\beta}{\eta}}(T-s)} -1}{\zeta e^{2\sqrt{\frac{a\beta}{\eta}}(T-s)} +1}ds.$$
Using that $\sqrt{\frac{a\beta}{\eta}}\int_t^T \frac{\zeta e^{2\sqrt{\frac{a\beta}{\eta}}(T-s)} }{\zeta e^{2\sqrt{\frac{a\beta}{\eta}}(T-s)} +1}ds = -\frac{1}{2}\ln \left( \frac{1+\zeta e^{2\sqrt{\frac{a\beta}{\eta}}(T-t)}}{1+\zeta}\right)$ and  $-\sqrt{\frac{a\beta}{\eta}}\int_t^T \frac{1 }{\zeta e^{2\sqrt{\frac{a\beta}{\eta}}(T-s)} +1}ds = \frac{1}{2}\ln \left( \frac{\zeta+ e^{-2\sqrt{\frac{a\beta}{\eta}}(T-t)}}{1+\zeta}\right)$, and isolating $y_0(t)$ yields 
$$ y_0(t) =  \frac{1+\zeta}{e^{\left(\kappa - \sqrt{\frac{a\beta}{\eta}} \right)(T-t)} +\zeta e^{\left(\kappa + \sqrt{\frac{a\beta}{\eta}} \right)(T-t)}}.$$
Now, by the variation of constants method, we easily obtain that the solution to the original equation is given by 
$$ y(t) = \left(-\frac{1}{2} - \frac{\kappa}{2} \int_t^T y_0(s)^{-1}ds \right) y_0(t)$$
which gives the claimed value for $p_{12,t} = y(t) + \frac{1}{2}.$ We finally need to show that $f_t$ = 0. Let $b_t = (b_{1,t}, b_{2,t})^T$. Then we have $f_t = \frac{b_{2,t}}{2\eta}$. Plugging this value into the differential equation for $b_{2,t}$ gives
$$ \partial_t b_{2,t} = \frac{2a}{\eta} b_{2,t} p_{22,t}$$
with the terminal condition $b_{2,T} = 0$. It is immediate to see that $b_{2,t}=0$ is the unique solution to the above Cauchy problem, so that $f_t =0$ as well and we are done. The corollary is simply obtained by first taking the pointwise limit $\beta \to 0$ in the above expressions and then by taking $\tilde{\beta}_T \to +\infty$.
\end{proof}

\subsection{Proofs of Theorem \ref{thmLQRDark} and Theorem \ref{thmLQRcontDark}}

The following lemma will be convenient for calculations.

\newtheorem{lemma}{Lemma}
\begin{lemma}
Let $P \in \mathbb{R}^{2 \times 2}$ be a deterministic matrix. Then for $0 \leq t \leq T$
$$ \mathbb{E}\left[\frac{a_t^2}{v_t^2} (\epsilon_t^D)^2 e_2^TPe_2\right] = \frac{a_t^2}{v_t^2} \lambda_t^D e_2^TPe_2 \delta t.$$
 
\end{lemma}

\begin{proof}
This is a direct consequence of the fact that $\epsilon_t^D$ follows a Bernoulli distribution with parameter $\lambda_t^D\delta t$.
\end{proof}

\begin{proof}[Proof of Theorem \ref{thmLQRDark}]
Let $0 \leq t \leq T$ and $s \geq t$ as in the theorem. Recall that 
$$ j_s^D(X^D,y) =   (\delta t)^{-1}\frac{a_s}{v_s}\theta_s^D \epsilon_s^D  y^2 + (\delta t)^{-1}\frac{a_s}{v_s}\epsilon_s^Dp y   +  \left(\theta_s^D+ (\delta t)^{-1}\frac{\eta_s}{v_s} \epsilon_s^D\right)a_s \langle u \rangle_t y_s + \beta_s^D (q^D)^2. $$
Let 
$$ h_s^D(X^D,y) = \mathbb E \left[ H^D\left(s+ \delta t,  X^D - \kappa_s \delta t \mathcal{J} X^D - \frac{a_s}{v_s} \epsilon_s^D e_2 y+  \left( \begin{matrix} \alpha_s \delta t + \sigma_{s, \delta t} z_s \\ 0 \end{matrix} \right) \right)\right].$$
Similar reasoning to that of the proof of Theorem \ref{thmLQR} yields
\begin{eqnarray*}
h_s^D(X,y) &=& (X^D)^T(\mathcal I - \kappa_s \delta t\mathcal J)P_{s + \delta t}^D   (\mathcal I - \kappa_s \delta t \mathcal J) X^D  \\
&-& 2 \frac{a_s\lambda_s^D}{v_s}(X^D)^T(\mathcal I - \kappa_s \delta t \mathcal J) P_{s + \delta t}^D e_2 y \delta t  +2\alpha_s \delta t e_1^T P_{s + \delta t}^D \left[ (\mathcal I - \kappa_s \delta t \mathcal J) X^D - \frac{a_s \lambda_s^D \delta t}{v_s} e_2 y \right]\\
&+& a_s^2\delta t  \frac{\lambda_s^D}{v_s^2} e_2^T P_{s + \delta t}^D e_2 y^2 + (b_{s + \delta t}^D)^T  [(\mathcal I - \kappa_s\delta t \mathcal J) X^D +\alpha_s \delta t e_1]  - a_s\frac{\lambda_s^D}{v_s} (b_{s + \delta t}^D)^T e_2 y \delta t\\ 
&+& e_{s + \delta t} + \sigma_{s,\delta t}^2 e_1^T P_{s + \delta t}^D e_1,
\end{eqnarray*}
hence
\begin{eqnarray*}
\mathbb E [j_s^D(X,y)\delta t] + h_s^D(X,y) &=&  \left[a_s \frac{\theta_s^D\lambda_s^D}{v_s} +  \frac{a_s^2\lambda_s^D}{v_s^2}e_2^T P_{s + \delta t}^De_2  \right]\delta t y^2 \\
&+&  \frac{a_s}{v_s} \lambda_s^D \bigg[ (X^D)^Te_1    -2 (X^D)^T(\mathcal I - \kappa_s\delta t \mathcal J) P_{s + \delta t}^D e_2  \\
&-&   (b_{s + \delta t}^D)^T e_2   - 2\alpha_s \delta t e_1^TP_{s + \delta t}^De_2+ \left(\frac{v_s\theta_s^D}{\lambda_s^D} + \eta_s \right) \langle u \rangle_t\bigg]y \delta t \\
&+& (X^D)^T[(\mathcal I - \kappa_s \delta t\mathcal J)  P_{s + \delta t}^D (\mathcal I - \kappa_s \delta t\mathcal J) + \Delta_s^D]X^D \\
&+& (b_{s + \delta t}^D)^T[(\mathcal I - \kappa_s\delta t \mathcal J)X^D+\alpha_s \delta t e_1]  \\&+&  2\alpha_s \delta t e_1^T P_{s + \delta t}^D  (\mathcal I - \kappa_s \delta t \mathcal J) X^D+ e_{s + \delta t}^D+[\sigma_{s,\delta t}^2+\alpha_s^2 (\delta t)^2] e_1^T P_{s + \delta t}^De_1
\end{eqnarray*}
where $\Delta_s^D = \text{diag}[0, \beta_s^D]$. Letting 
$$ g_s^D = \theta_s^D   + \frac{ a_s}{v_s} e_2^T P_{s + \delta t}^De_2  $$

$$(k_s^D)^T = -  (g_s^D)^{-1}    \left[ -e_2^T (P_{s + \delta t}^D)^T (\mathcal I -\kappa_s \delta t \mathcal J) + \frac{1}{2} e_1^T \right]   $$
and
$$ f_s^D = -\frac{1}{2} (g_s^D)^{-1} \left[-e_2^Tb_{s + \delta t}^D- 2\alpha_s \delta t e_1^TP_{s + \delta t}^De_2 +\left(\frac{v_s\theta_s^D}{\lambda_s^D} + \eta_s\right) \langle u\rangle_t\right] $$
we obtain

\begin{eqnarray*}
\mathbb E [j_s^D(X^D,y)\delta t] + h_s(X^D,y) &=& \frac{a_s\lambda_s^D}{v_s} \delta t g_s^D y^2 -2 \frac{a_s\lambda_s^D}{v_s}\delta t [ X^T k_s^D +  f_s^D] g_s^D y \\
&+& (X^D)^T[(\mathcal I - \kappa_s \delta t \mathcal J)  P_{s + \delta t}^D (\mathcal I - \kappa_s\delta t \mathcal J) + \Delta_s^D]X^D \\
&+& (b_{s + \delta t}^D)^T(\mathcal I - \kappa_s\delta t \mathcal J)X^D + e_{s + \delta t}^D+[\sigma_{s,\delta t}^2+\alpha_s^2 (\delta t)^2]  e_1^T P_{s + \delta t}^De_1 
\end{eqnarray*}

hence the expression of $y_s^*$ in feedback form 
$$ y_s^* = (k_s^D)^T X^D + f_s^D.$$

Plugging back $y_s^*$ into the expression above yields
 
$$ P_s^D =  P_{s + \delta t}^D - \kappa_s \delta t (\mathcal J P_{s + \delta t}^D + P_{s + \delta t}^D \mathcal J) + \kappa_s^2 (\delta t)^2 \mathcal J P_{s + \delta t}^D \mathcal J + \left[\Delta_s  - \frac{a_s \lambda_s^D}{v_s}g_s^D k_s^D  (k_s^D)^T\right]\delta t,$$
$$ b_s^D = b_{s + \delta t}^D -\left\{2\frac{a_s \lambda_s^D}{v_s} g_s^D f_s^D k_s^D  + \kappa_s \mathcal Jb_{s + \delta t}  -2\alpha_s (\mathcal I - \kappa_s\delta t \mathcal J)  P_{s + \delta t}^De_1 \right\}\delta t,$$
$$e_s^D = e_{s + \delta t}^D + [\sigma_{s,\delta s}^2+\alpha_s^2 (\delta t)^2]  e_1^T P_{s + \delta t}^De_1   +\left[\alpha_s (b_{s + \delta t}^D)^Te_1 - \frac{a_s \lambda_s^D}{v_s}   g_s^D  (f_s^D)^2\right] \delta t $$
as claimed in the theorem.

\end{proof}

\begin{proof}[Proof of Theorem \ref{thmLQRcontDark}]
This follows the same line of reasoning as for the proof of Theorem \ref{thmLQRContinuous}.
\end{proof}

\bibliography{biblio}

\begin{thebibliography}{25}
\providecommand{\natexlab}[1]{#1}
\providecommand{\url}[1]{\texttt{#1}}
\expandafter\ifx\csname urlstyle\endcsname\relax
  \providecommand{\doi}[1]{doi: #1}\else
  \providecommand{\doi}{doi: \begingroup \urlstyle{rm}\Url}\fi

\bibitem[Alessio and Bemporad(2009)]{alessio2009survey}
Alessandro Alessio and Alberto Bemporad.
\newblock A survey on explicit model predictive control.
\newblock In \emph{Nonlinear model predictive control}, pages 345--369.
  Springer, 2009.

\bibitem[Alfonsi and Blanc(2016)]{alfonsi2016dynamic}
Aur{\'e}lien Alfonsi and Pierre Blanc.
\newblock Dynamic optimal execution in a mixed-market-impact hawkes price
  model.
\newblock \emph{Finance and Stochastics}, 20\penalty0 (1):\penalty0 183--218,
  2016.

\bibitem[Almgren and Chriss(2001)]{almgren2001optimal}
Robert Almgren and Neil Chriss.
\newblock Optimal execution of portfolio transactions.
\newblock \emph{Journal of Risk}, 3:\penalty0 5--40, 2001.

\bibitem[Bemporad and Filippi(2003)]{bemporad2003suboptimal}
Alberto Bemporad and Carlo Filippi.
\newblock Suboptimal explicit receding horizon control via approximate
  multiparametric quadratic programming.
\newblock \emph{Journal of optimization theory and applications}, 117\penalty0
  (1):\penalty0 9--38, 2003.

\bibitem[Bemporad et~al.(2002)Bemporad, Morari, Dua, and
  Pistikopoulos]{bemporad2002explicit}
Alberto Bemporad, Manfred Morari, Vivek Dua, and Efstratios~N Pistikopoulos.
\newblock The explicit linear quadratic regulator for constrained systems.
\newblock \emph{Automatica}, 38\penalty0 (1):\penalty0 3--20, 2002.

\bibitem[Bertsekas et~al.(1995)Bertsekas, Bertsekas, Bertsekas, and
  Bertsekas]{bertsekas1995dynamic}
Dimitri~P Bertsekas, Dimitri~P Bertsekas, Dimitri~P Bertsekas, and Dimitri~P
  Bertsekas.
\newblock \emph{Dynamic programming and optimal control}, volume~1.
\newblock Athena scientific Belmont, MA, 1995.

\bibitem[Bertsimas and Lo(1998)]{bertsimas1998optimal}
Dimitris Bertsimas and Andrew~W Lo.
\newblock Optimal control of execution costs.
\newblock \emph{Journal of Financial Markets}, 1\penalty0 (1):\penalty0 1--50,
  1998.

\bibitem[Busseti and Boyd(2015)]{busseti2015volume}
Enzo Busseti and Stephen Boyd.
\newblock Volume weighted average price optimal execution.
\newblock \emph{arXiv preprint arXiv:1509.08503}, 2015.

\bibitem[Cartea et~al.(2015)Cartea, Jaimungal, and
  Penalva]{cartea2015algorithmic}
{\'A}lvaro Cartea, Sebastian Jaimungal, and Jos{\'e} Penalva.
\newblock \emph{Algorithmic and high-frequency trading}.
\newblock Cambridge University Press, 2015.

\bibitem[Cutler and Ramaker(1980)]{cutler1980dynamic}
Charles~R Cutler and Brian~L Ramaker.
\newblock Dynamic matrix control - a computer control algorithm.
\newblock In \emph{joint automatic control conference}, number~17, page~72,
  1980.

\bibitem[Garcia et~al.(1989)Garcia, Prett, and Morari]{garcia1989model}
Carlos~E Garcia, David~M Prett, and Manfred Morari.
\newblock Model predictive control: theory and practice—a survey.
\newblock \emph{Automatica}, 25\penalty0 (3):\penalty0 335--348, 1989.

\bibitem[Gatheral and Schied(2011)]{gatheral2011optimal}
Jim Gatheral and Alexander Schied.
\newblock Optimal trade execution under geometric brownian motion in the
  almgren and chriss framework.
\newblock \emph{International Journal of Theoretical and Applied Finance},
  14\penalty0 (03):\penalty0 353--368, 2011.

\bibitem[Gu{\'e}ant(2016)]{gueant2016financial}
Olivier Gu{\'e}ant.
\newblock \emph{The Financial Mathematics of Market Liquidity: From optimal
  execution to market making}, volume~33.
\newblock CRC Press, 2016.

\bibitem[Gu{\'e}ant and Royer(2014)]{gueant2014vwap}
Olivier Gu{\'e}ant and Guillaume Royer.
\newblock Vwap execution and guaranteed vwap.
\newblock \emph{SIAM Journal on Financial Mathematics}, 5\penalty0
  (1):\penalty0 445--471, 2014.

\bibitem[Hora(2006)]{hora2006tactical}
Merell Hora.
\newblock Tactical liquidity trading and intraday volume.
\newblock \emph{Preprint}, 2006.

\bibitem[Huberman and Stanzl(2004)]{huberman2004price}
Gur Huberman and Werner Stanzl.
\newblock Price manipulation and quasi-arbitrage.
\newblock \emph{Econometrica}, 72\penalty0 (4):\penalty0 1247--1275, 2004.

\bibitem[Huberman and Stanzl(2005)]{huberman2005optimal}
Gur Huberman and Werner Stanzl.
\newblock Optimal liquidity trading.
\newblock \emph{Review of Finance}, 9\penalty0 (2):\penalty0 165--200, 2005.

\bibitem[Li et~al.(2014)Li, Xia, Su, Deng, Fu, and He]{li2014missile}
Zhijun Li, Yuanqing Xia, Chun-Yi Su, Jun Deng, Jun Fu, and Wei He.
\newblock Missile guidance law based on robust model predictive control using
  neural-network optimization.
\newblock \emph{IEEE transactions on neural networks and learning systems},
  26\penalty0 (8):\penalty0 1803--1809, 2014.

\bibitem[Martin et~al.(1986)Martin, Caldwell, and Ayral]{martin1986predictive}
GD~Martin, JM~Caldwell, and TE~Ayral.
\newblock Predictive control applications for the petroleum refining industry.
\newblock In \emph{Petroleum Refining Conference, Tokyo, Japan}, 1986.

\bibitem[Morari and Lee(1999)]{morari1999model}
Manfred Morari and Jay~H Lee.
\newblock Model predictive control: past, present and future.
\newblock \emph{Computers \& Chemical Engineering}, 23\penalty0 (4-5):\penalty0
  667--682, 1999.

\bibitem[Obizhaeva and Wang(2013)]{obizhaeva2013optimal}
Anna~A Obizhaeva and Jiang Wang.
\newblock Optimal trading strategy and supply/demand dynamics.
\newblock \emph{Journal of Financial Markets}, 16\penalty0 (1):\penalty0 1--32,
  2013.

\bibitem[Qin and Badgwell(2003)]{qin2003survey}
S~Joe Qin and Thomas~A Badgwell.
\newblock A survey of industrial model predictive control technology.
\newblock \emph{Control engineering practice}, 11\penalty0 (7):\penalty0
  733--764, 2003.

\bibitem[Richalet et~al.(1978)Richalet, Rault, Testud, and
  Papon]{richalet1978model}
Jacques Richalet, Andr{\'e} Rault, JL~Testud, and J~Papon.
\newblock Model predictive heuristic control.
\newblock \emph{Automatica (journal of IFAC)}, 14\penalty0 (5):\penalty0
  413--428, 1978.

\bibitem[Schied et~al.(2010)Schied, Sch{\"o}neborn, and
  Tehranchi]{schied2010optimal}
Alexander Schied, Torsten Sch{\"o}neborn, and Michael Tehranchi.
\newblock Optimal basket liquidation for cara investors is deterministic.
\newblock \emph{Applied Mathematical Finance}, 17\penalty0 (6):\penalty0
  471--489, 2010.

\bibitem[Shen(2017)]{shen2017hybrid}
Jackie Shen.
\newblock Hybrid is-vwap dynamic algorithmic trading via lqr.
\newblock \emph{Available at SSRN 2984297}, 2017.

\end{thebibliography}

\end{document}